\documentclass[11pt,a4paper]{article}

\usepackage{amsmath}
\usepackage{amssymb,latexsym}
\usepackage{mathabx}
\usepackage{amsthm}
\usepackage[shortlabels]{enumitem}
\usepackage{array}
\usepackage[hang,flushmargin,bottom]{footmisc}
\usepackage[margin=0.35in,format=hang,font=small,labelfont=bf]{caption}
\usepackage[normalem]{ulem}
\usepackage{needspace}
\usepackage{graphicx,tikz}
\usepackage[colorlinks=true,urlcolor=blue,linkcolor=blue,citecolor=blue]{hyperref}
\usepackage{palatino}
\usepackage{eulervm}
\usepackage{booktabs}

\usepackage{ifthen}
\usetikzlibrary{calc}
\usetikzlibrary{decorations.pathreplacing,decorations.markings}

\newcommand{\plotptradius}{4pt}
\newcommand{\setplotptradius}[1]{\renewcommand{\plotptradius}{#1}}

\tikzset{permpt/.style={circle, draw, fill=black, inner sep=0pt, minimum width=\plotptradius}}
\tikzset{hollow/.style={circle, draw, fill=none, inner sep=0pt, minimum width=\plotptradius}}
\tikzset{empty/.style={draw=none, fill=none}}

\newcommand\absdot[2]{
	\node[permpt] at #1 {};
}

\newcommand{\plotperm}[2][black]{ 
	\foreach \j [count=\i] in {#2} {
        \ifnum0=\j {} \else {
 		\node[permpt,fill=#1,draw=#1] (\j) at (\i,\j) {};
	} \fi
	}
}
\newcommand{\plotpermbox}[4]{
	\draw [darkgray, very thick, line cap=round, fill=white]
		({#1-0.5}, {#2-0.5}) rectangle ({#3+0.5}, {#4+0.5});
}

\newcommand{\plotpermborder}[1]{
	\foreach \i [count=\nn] in {#1} {\global\let\n\nn};    
	\plotpermbox{1}{1}{\n}{\n};
	\plotperm{#1};
}
\newcommand{\plotpermbordergrid}[1]{
	\foreach \i [count=\nn] in {#1} {\global\let\n\nn};    
	\plotpermbox{1}{1}{\n}{\n};
	\draw[step=1cm,gray!50,very thin] (1,1) grid (\n,\n);
	\plotperm{#1};
}
\newcommand{\plotgrid}[1]{
	\draw[step=1cm,gray!50,very thin] (0.5,0.5) grid (#1+0.5,#1+0.5);
}

\newcommand{\plotpermgrid}[1]{
	\foreach \i [count=\nn] in {#1} {\global\let\n\nn};    
	\plotgrid{\n};
	\plotperm{#1};
}

\newcommand{\plotpinsequence}[1]{
	\absdot{(0,0)}{};
	\edef\n{0}
	\edef\s{0}
	\edef\e{0}
	\edef\w{0}
	\edef\x{0}
	\edef\y{0}
	\foreach \pin [remember=\pin as \oldpin (initially 1), count=\i] in {#1} {
		\ifthenelse{\pin=1 \OR \pin=2}{
			\ifthenelse{\oldpin=3}{
				\xdef\x{\number\numexpr\e-1}
			}{
				\xdef\x{\number\numexpr\w+1}
			}
			\ifnum\i=1 
				\pgfmathparse{\e+1}
 				\xdef\e{\pgfmathresult}
			\fi	
		}{ 
			\ifthenelse{\oldpin=1}{
				\xdef\y{\number\numexpr\n-1}
			}{
				\xdef\y{\number\numexpr\s+1}
			}
			\ifnum\i=1 
				\pgfmathparse{\s-1}
 				\xdef\s{\pgfmathresult}
			\fi	
		}
		\ifnum\pin=1 
			\pgfmathparse{\n+2}
 			\xdef\n{\pgfmathresult}		
			\absdot{(\x,\n)}{};
			\ifnum\i>1
				\draw (\x,\n) -- (\x,\y-0.5);
			\else
				\draw[gray,very thick] (-0.5,-0.5) rectangle (\x+0.5,\n+0.5);
			\fi
		\fi
		\ifnum\pin=2 
			\pgfmathparse{\s-2}
 			\xdef\s{\pgfmathresult}
			\absdot{(\x,\s)}{};
			\ifnum\i>1
				\draw (\x,\s) -- (\x,\y+0.5);
			\else
				\draw[gray,very thick] (-0.5,0.5) rectangle (\x+0.5,\s-0.5);
			\fi
		\fi
		\ifnum\pin=3 
			\pgfmathparse{\e+2}
 			\xdef\e{\pgfmathresult}
			\absdot{(\e,\y)}{};
			\ifnum\i>1
				\draw (\e,\y) -- (\x-0.5,\y);
			\else
				\draw[gray,very thick] (-0.5,+0.5) rectangle (\e+0.5,\y-0.5);
			\fi
		\fi
		\ifnum\pin=4 
			\pgfmathparse{\w-2}
 			\xdef\w{\pgfmathresult}
			\absdot{(\w,\y)}{};
			\ifnum\i>1
				\draw (\w,\y) -- (\x+0.5,\y);
			\else
				\draw[gray,very thick] (0.5,0.5) rectangle (\w-0.5,\y-0.5);

			\fi
		\fi		
	};
}

\tikzset{
  on each segment/.style={
    decorate,
    decoration={
      show path construction,
      moveto code={},
      lineto code={
        \path [#1]
        (\tikzinputsegmentfirst) -- (\tikzinputsegmentlast);
      },
      curveto code={
        \path [#1] (\tikzinputsegmentfirst)
        .. controls
        (\tikzinputsegmentsupporta) and (\tikzinputsegmentsupportb)
        ..
        (\tikzinputsegmentlast);
      },
      closepath code={
        \path [#1]
        (\tikzinputsegmentfirst) -- (\tikzinputsegmentlast);
      },
    },
  },
  mid arrow/.style={postaction={decorate,decoration={
        markings,
        mark=at position .5 with {\arrow[#1]{stealth}}
      }}},
}
\usetikzlibrary{decorations,patterns}

\newtheorem{theorem}{Theorem}[section]
\newtheorem*{theorem*}{Theorem}
\newtheorem{proposition}[theorem]{Proposition}
\newtheorem{corollary}[theorem]{Corollary}

\newtheorem*{conj*}{Conjecture}

\theoremstyle{definition}

\newtheorem{defn*}{Definition}

\newtheorem*{example*}{Example}

\newtheorem*{comment*}{Comment}

\newcommand{\AAA}{\mathcal{A}}
\newcommand{\BBB}{\mathcal{B}}
\newcommand{\CCC}{\mathcal{C}}

\newcommand{\DDD}{\mathcal{D}}

\newcommand{\HHH}{\mathcal{H}}

\newcommand{\RRR}{\mathcal{R}}
\newcommand{\SSS}{\mathcal{S}}

\newcommand{\av}{\mathsf{Av}}
\newcommand{\gr}{\mathrm{gr}}

\newcommand{\disc}{\mathrm{disc}}

\newcommand{\lr}{{\color{red!50!black}\mathsf{R}}}
\newcommand{\lro}{{\color{red!50!black}\mathring{\lr}}}
\newcommand{\lb}{{\color{blue!50!black}\mathsf{B}}}
\newcommand{\lbo}{{\color{blue!50!black}\mathring{\lb}}}

\tikzstyle{vertex}=[circle, draw, fill=black,
                        inner sep=0pt, minimum width=4pt]

\newcommand{\blueArc}[2] {\draw [blue!50!black,very thick] (#1,1.25) arc (0:180:#2);} 
\newcommand{\fillBlueArc}[2] {\fill [gray!40!white] (#1,1.25) arc (0:180:#2); \draw [gray!40!white,line width=6] (#1,1)--(#1-2*#2,1);} 
\newcommand{\bluePt}[1]  {\draw [blue!50!black,very thick] (#1,1.25)--(#1,1.85);}
\newcommand{\redArc}[2]{\draw [red!50!black,very thick] (#1,0.75) arc (180:360:#2);} 
\newcommand{\redPt}[1] {\draw [red!50!black,very thick] (#1,0.75)--(#1,0.15);}
\newcommand{\archavoid}{
\!%
\raisebox{-2pt}{
  \begin{tikzpicture}[scale=0.15,line join=round]
    \setplotptradius{2pt}
    \draw [gray,thin] (0,1)--(5,1);
    \redArc{1}{1 and .8}
    \blueArc{4}{1 and .8}
    \plotperm[blue!50!black]{0,1,0,1}
    \plotperm[red!50!black]{1,0,1,0}
  \end{tikzpicture}
}%
\!
}

\title{Shuffle automata and the growth of 1324-avoiding permutations}

\author{Robert Brignall\\
\small School of Mathematics and Statistics\\
\small The Open University, UK
}

\begin{document}
\maketitle

\begin{abstract}
We show that $\gr(\av(1324))\leq 13.167248$. This is achieved by combining the approach of Bevan, Brignall, Elvey Price and Pantone using the `domino', with ideas from earlier upper bounds by B\'ona using pairs of decorated words with additional restrictions. More specifically, we replace the arbitrary interleavings of Bevan et. al. with interleavings more like those of B\'ona. Our method uses shuffle automata to handle the interleavings, and this method has the potential to provide better upper bounds than the one established here. To estimate how many shuffles are possible, we use two key statistics on dominoes: the number of internal points (that is, points that are neither left-to-right minima not right-to-left maxima), and the number of runs of internal points.
\end{abstract}

%
%
%
%
%
%
%
%
\section{Introduction}

The class $\av(1324)$ is infamously difficult to count. Up to symmetry, there are seven distinct classes $\av(\beta)$ where $\beta$ is a permutation of length 4, and the precise enumeration of the other six was completed in the late 1990s. On the other hand, much about $\av(1324)$ remains unknown, including the asymptotic growth rate. Numerical analysis based on extensive computation by Conway, Guttmann and Zinn-Justin~\cite{conway:1324-revisited} gives an estimate of $\gr(\av(1324))=11.600\pm0.003$, and they suggest an unusual form for its first-order asymptotics involving a `stretched exponential' term. However, a substantial gap remains between this estimate and rigorous bounds.

The best published upper and lower bounds on the growth rate are due to the author with Bevan, Elvey Price and Pantone~\cite{bbep:1324}. In that paper, the `staircase decomposition' of $\av(1324)$ is introduced and used to establish the precise enumeration of an associated class known as the `domino'. We review the salient points of this paper in Sections~\ref{sec:preliminaries} and~\ref{sec:weighted-dominoes}, but for now it suffices to note that the growth rate of this `domino' class is $27/4 = 6.75$. By exploiting knowledge about the domino class, the following bounds are obtained:
\[
10.271 \leq  \gr(\av(1324))\leq 13.5.
\]
A recent preprint by Norton~\cite{norton:a-new-lower-bound} claims a new lower bound of 10.412263, by combining the method of~\cite{bbep:1324} with the \emph{Harris inequality}~\cite{harris:a-lower-bound}.

Our focus here is on upper bounds. The upper-bound argument of~\cite{bbep:1324} encodes a $1324$-avoider by a domino (whose growth rate is $27/4$) together with a binary word recording how the two colours are interleaved in value order (of `growth rate' $2$). It is remarked there that the use of the arbitrary binary word is `rudimentary'; our aim here is to develop a more sophisticated method whilst following the same basic approach.  We use ideas from two earlier record upper bounds due to B\'ona~\cite{bona:a-new-upper-bou:,bona:a-new-record-for:}, in which permutations are encoded by two words over an alphabet of four letters.  Here, we use the domino, plus one word over a four-letter alphabet to record the interleavings by value. This word is essentially a shuffle of two words (over two disjoint two-letter alphabets) that satisfies certain additional restrictions. By counting these shuffles using `shuffle automata' and combining them with a refinement of the domino enumeration to track two additional statistics, we develop a framework in which automata can be used to find improved bounds on $\gr(\av(1324))$. The bound we present here is $\gr(\av(1324))\leq 13.167248$.

Note that there remains open a conjecture due to Claesson, Jel\'inek and Steingr\'imsson~\cite{claesson:upper-bounds-fo:} regarding
the number of 1324-avoiders of each length that have a fixed number of inversions. If true, it would imply an upper bound of $e^{\pi\sqrt{2/3}}\approx 13.002$.  Linusson and Verkama~\cite{linusson:enumerating-1324} made some progress towards this conjecture, but the general case (and the resulting upper bound on $\gr(\av(1324))$) remain out of reach for the time being. See also Claesson, Linusson, Ulfarsson and Verkama~\cite{claesson:inversion-monotonicity}.

\paragraph{Declaration of AI use}
The work in this paper was conducted by the author in conjunction with
ChatGPT 5.6 Sol and ChatGPT 6.0 Astra. The initial prompt asked a single
instance of ChatGPT to investigate whether the methods of~\cite{bbep:1324}
could be combined with those of B\'ona~\cite{bona:a-new-record-for:} to obtain
improvements on the upper bound. Its response indicated that this was
possible and claimed a new upper bound; importantly, this response identified the two key statistics and much of the reasoning required.
Subsequent prompts were used to improve on its initial result, to refine
the exposition, to prepare the formal verification code in Lean and to make deeper explorations beyond the content of the paper (see the concluding remarks). 
The mathematical content of this paper has been fully digested\footnote{See Tao~\cite{tao2026mathematicsageai} to understand the workflow used to prepare this paper.} by the author, who accepts full responsibility for
the results herein.

\paragraph{Organisation of this paper}
Section~\ref{sec:preliminaries} recalls the necessary terminology concerning permutation classes, the staircase decomposition of $\av(1324)$ and the domino, and then introduces the
decorated red--blue words used throughout the paper. It concludes with a brief overview of the method deployed in the rest of the paper. Section~\ref{sec:weighted-dominoes} uses essentially the same methods as~\cite{bbep:1324} to derive a weighted generating function
for dominoes that records the two statistics required (namely, internal points and runs of internal points).
Section~\ref{sec:shuffle-automata} develops the weighted shuffle-automaton machinery used to bound the number of admissible shuffles, and Section~\ref{sec:lookahead} builds two `lookahead' shuffle automata: the first of these is a toy example included primarily to aid exposition, while the second is the automaton that gives rise to the advertised upper bound. Section~\ref{sec:conclusion} provides various concluding remarks, including indicative results of various explorations on further refinements that have already been tested.

There are two appendices containing data for the analysis of the two automata introduced in Section~\ref{sec:lookahead}. Additionally, a GitHub repository exists with a Lean formalisation of Proposition~\ref{prop:three-step-certificate}. Note: it is only the arguments to prove Proposition~\ref{prop:three-step-certificate} that have been formalised, as this is the one step which would be neither rewarding nor informative were it to be verified by traditional means. The repository is at \url{https://github.com/rbrignall/1324-lookahead/}.

%
%
%
%
%
%
%
%
%
\section{Preliminaries}
\label{sec:preliminaries}

\subsection{Permutations, classes and the staircase}

A permutation of length $n$ is a bijection $\pi:[n]\to[n]$, which we
usually write in one-line notation and identify with its \emph{plot}
\[
 \{(i,\pi(i)):i\in[n]\}.
\]
We say that a permutation $\pi$ is \emph{contained} in $\sigma$ if some subsequence of $\sigma$ is order isomorphic to $\pi$; such a subsequence is an \emph{occurrence} of $\pi$.
If $\sigma$ does not contain $\pi$, then $\sigma$ \emph{avoids} $\pi$. 
We write $\av(\pi)$ to denote the set of permutations avoiding $\pi$, and $\av_n(\pi)$ for the permutations in $\av(\pi)$ of length $n$.

The sets $\av(\pi)$ for different values of $\pi$ are \emph{permutation classes}, which more generally are collections of permutations that are downwards closed under containment. If $\CCC$ is a permutation class, let $\CCC_n$ denote
its members of length $n$, and define its exponential growth rate by
\[
 \gr(\CCC)=\limsup_{n\to\infty}\sqrt[n]{|\CCC_n|}.
\]
For every proper permutation class, this limsup is finite by Marcus and Tardos~\cite{marcus:excluded-permut:}. For a principal class $\av(\pi)$, Arratia~\cite{arratia:on-the-stanley-:} used Fekete's lemma to show that limsup can be replace by lim. 

Our work considers a number of different classes of combinatorial structures. In each case, we will as required extend the term \emph{growth rate} to refer to $\limsup_{n\to\infty}\sqrt[n]{a_n}$, where $a_n$ denotes the number of these objects of `size' $n$.

The class of interest in this paper is, of course, $\av(1324)$, but there are two other rather more tractable classes that feature: $\av(132)$ and $\av(213)$. Bevan, Brignall, Elvey Price and Pantone~\cite{bbep:1324} showed that every permutation in $\av(1324)$ admits a decomposition into the descending $(\av(213),\av(132))$ \emph{staircase}, which is
\[
 \begin{matrix}
  \av(213) & \\
  \av(132) & \av(213)    \\
            & \av(132)   & \av(213)    \\
            &            & \ddots      & \ddots .
 \end{matrix}
\]
In other words, the plot of every permutation $\pi\in\av(1324)$ can be partitioned via a sequence of vertical and horizontal dividing lines so that all the points of $\pi$ belong to the cells on the diagonal and sub-diagonal, and so that the permutations formed by the points of $\pi$ in each cell belong to the class of the corresponding entry in the above diagram. Throughout, we use the \emph{greedy} staircase decomposition of~\cite{bbep:1324}, which sequentially takes each cell to contain as many points as possible, from top left to bottom right.

Given the plot of a permutation $\pi\in\av(1324)$ equipped with a greedy staircase decomposition, it is helpful to use colours to distinguish between the points in the two types of cell. Specifically, colour the points in $\av(213)$ cells blue, and those in $\av(132)$ cells red.   Note that all the blue points collectively avoid $213$, while the red points collectively avoid $132$. Furthermore, no red point can lie both above and to the right of a blue
point: that is, if a red point lies above a blue point, then it lies to its left.

%
%
%
%
%
\subsection{Dominoes and decorated red--blue words}

A \emph{domino} is a permutation divided by a single horizontal line into two cells, such that the upper cell avoids 213 and the lower 132, and whose underlying permutation avoids 1324. Either cell may be empty, and it is only the resulting partition of the points between the two cells that is retained (rather than the precise position of the dividing line). In other words, the dominoes are the permutations of $\av(1324)$ that can be drawn into the top two cells of the staircase decomposition, together with the line that distinguishes the two cells. %
Let $\DDD$ denote the set of dominoes, and $\DDD_n$ those with $n$ points.

Given the greedy staircase decomposition of a permutation $\pi\in\av(1324)$, the upper-bound
construction of~\cite{bbep:1324} produces a domino $\delta(\pi)$ by
placing all blue points above all red points, while preserving horizontal
order and the relative vertical order within each colour. It also records
a binary word $\chi(\pi)$ which, reading the points of $\pi$ from bottom-to-top, records whether each point is blue or red. See Figure~\ref{figBBEPUpperBound}.

\begin{figure}[ht]
  \centering
    \setplotptradius{2pt}
    \begin{tikzpicture}[scale=0.2]
      \fill[blue!5] (3.5,26.5) rectangle (11.5,23.5);
      \fill[red!10] (3.5,23.5) rectangle (11.5,13.5);
      \fill[blue!5] (11.5,23.5) rectangle (20.5,13.5);
      \fill[red!10] (11.5,13.5) rectangle (20.5,5.5);
      \fill[blue!5] (20.5,13.5) rectangle (27.5,5.5);
      \fill[red!10] (20.5,5.5) rectangle (27.5,2.5);
      \plotperm[blue!50!black,thick]{ 0,  0,  0, 24,  0,  0, 26,  0, 25,  0,  0, 20, 14,  0,  0, 0, 17,  0, 16, 15, 9, 0, 11, 0, 8, 7, 0}
      \plotperm[ red!50!black,thick]{ 0,  0,  0,  0, 23, 19,  0, 18,  0, 21, 22,  0,  0, 10, 12, 6,  0, 13,  0,  0, 0, 5,  0, 3, 0, 0, 4}
      \draw[thin] (3.5,26.5) -- (11.5,26.5);
      \draw[thin] (3.5,13.5) -- (3.5,26.5);
      \draw[thin] (3.5,23.5) -- (20.5,23.5);
      \draw[thin] (11.5,5.5) -- (11.5,26.5);
      \draw[thin] (3.5,13.5) -- (27.5,13.5);
      \draw[thin] (20.5,2.5) -- (20.5,23.5);
      \draw[thin] (11.5,5.5) -- (27.5,5.5);
      \draw[thin] (27.5,2.5) -- (27.5,13.5);
      \draw[thin] (20.5,2.5) -- (27.5,2.5);
      \node at (33,14.5) {$\mapsto$};
    \end{tikzpicture}
    $\quad$
    \begin{tikzpicture}[scale=0.2,>=stealth]
      \fill[blue!5] (0.5,24.5) rectangle (8.5,21.5);
      \fill[blue!5] (8.5,21.5) rectangle (17.5,16.5);
      \fill[blue!5] (17.5,16.5) rectangle (24.5,12.5);
      \fill[red!10] (0.5,12.5) rectangle (8.5,7.5);
      \fill[red!10] (8.5,7.5) rectangle (17.5,3.5);
      \fill[red!10] (17.5,3.5) rectangle (24.5,0.5);
      \draw[thin,gray] (.5,21.5)--(17.5,21.5);
      \draw[thin,gray] (8.5,16.5)--(24.5,16.5);
      \draw[thin,gray] (.5,7.5)--(17.5,7.5);
      \draw[thin,gray] (8.5,3.5)--(24.5,3.5);
      \draw[thin,gray] (8.5,3.5)--(8.5,12.5);
      \draw[thin,gray] (8.5,16.5)--(8.5,24.5);
      \draw[thin,gray] (17.5,0.5)--(17.5,7.5);
      \draw[thin,gray] (17.5,12.5)--(17.5,21.5);
      \draw[thick] (.5,.5) rectangle (24.5,24.5);
      \draw[thick] (.5,12.5)--(24.5,12.5);
      \draw[blue!50!white,thick,->] ( 9,21-0.25)->( 9,18);
      \draw[blue!50!white,thick,->] (14,20-0.25)->(14,15);
      \draw[blue!50!white,thick,->] (16,19-0.25)->(16,14);
      \draw[blue!50!white,thick,->] (17,18-0.25)->(17,13);
      \draw[blue!50!white,thick,->] (10,17-0.25)->(10,12);
      \draw[blue!50!white,thick,->] (20,16-0.25)->(20, 9);
      \draw[blue!50!white,thick,->] (18,15-0.25)->(18, 7);
      \draw[blue!50!white,thick,->] (22,14-0.25)->(22, 6);
      \draw[blue!50!white,thick,->] (23,13-0.25)->(23, 5);
      \draw[red!50,thick,->] ( 2,12)->( 2,21);
      \draw[red!50,thick,->] ( 8,11)->( 8,20);
      \draw[red!50,thick,->] ( 7,10)->( 7,19);
      \draw[red!50,thick,->] ( 3, 9)->( 3,17);
      \draw[red!50,thick,->] ( 5, 8)->( 5,16);
      \draw[red!50,thick,->] (15, 7)->(15,11);
      \draw[red!50,thick,->] (12, 6)->(12,10);
      \draw[red!50,thick,->] (11, 5)->(11, 8);
      \plotperm[blue!50!black,thick]{22,  0, 0, 24, 0, 23,  0,  0, 21, 17, 0, 0, 0, 20, 0, 19, 18, 15, 0, 16, 0, 14, 13, 0}
      \plotperm[ red!50!black,thick]{ 0, 12, 9,  0, 8,  0, 10, 11,  0,  0, 5, 6, 4,  0, 7,  0,  0,  0, 3,  0, 1,  0,  0, 2}
      \tikzset{r/.style={red!50!black}}
      \tikzset{b/.style={blue!50!black}}
      \draw[thick] (29.5,.5) rectangle ++(1,24);
      \foreach \col [count=\n] in {r,r,r,r,b,b,b,r,b,r,r,b,b,b,b,r,r,b,r,r,r,b,b,b}
      	\node[permpt,\col,thick] at (30,\n) {};
    \end{tikzpicture}\par
\caption{The mapping used in~\cite{bbep:1324} to map a greedy-gridded 1324-avoider to a domino and a binary word.}
\label{figBBEPUpperBound}
\end{figure}
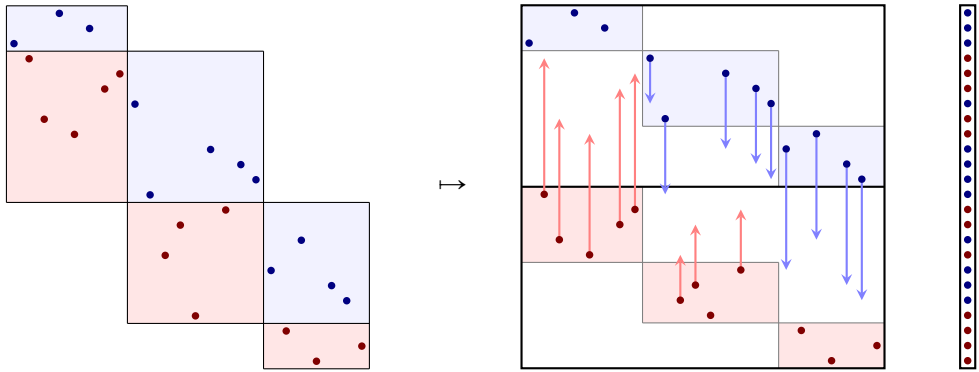

The resulting map
\[
 \pi\longmapsto\bigl(\delta(\pi),\chi(\pi)\bigr)
\]
is injective. Since there are $2^n$ binary words of length $n$, and
$\gr(\DDD)=27/4$ by~\cite[Theorem~3.1]{bbep:1324}, this gives
\[
 |\av_n(1324)|\leq 2^n|\DDD_n|
 \qquad\text{and hence}\qquad
 \gr(\av(1324))\leq\frac{27}{2}.
\]
The loss in this estimate is that, for a fixed domino, not every binary
word can arise as the vertical interleaving of a $1324$-avoiding
permutation. Our aim is to exploit some of the restrictions not captured in the above analysis.

Following B\'ona~\cite{bona:a-new-upper-bou:,bona:a-new-record-for:}, we
refine the two-letter alphabet by distinguishing extremal points within
each colour. A red point is assigned type $\lro$ if it is a left-to-right
minimum among the red points, and type $\lr$ otherwise. Similarly, a blue
point is assigned type $\lbo$ if it is a right-to-left maximum among the
blue points, and type $\lb$ otherwise. We call the decorated letters
$\lro$ and $\lbo$ \emph{external}, and the undecorated letters $\lr$ and
$\lb$ \emph{internal}.\footnote{To translate between this notation and
that of B\'ona~\cite{bona:a-new-upper-bou:}, use
$A=\lro$, $B=\lr$, $C=\lb$, $D=\lbo$.}
For convenience, put
\[
 \BBB=\{\lbo,\lb\},
 \qquad
 \RRR=\{\lro,\lr\}.
\]

Associated with $\pi$ are two words over $\RRR\cup\BBB$. The
\emph{position word} $w_{\mathrm p}(\pi)$ records the decorated letters from
left to right, while the \emph{value word} $w_{\mathrm v}(\pi)$ records
them from bottom to top. Since the construction of $\delta(\pi)$
preserves position order and the relative value order within each colour,
it also preserves the left-to-right minima among red points and the
right-to-left maxima among blue points. Hence $w_{\mathrm p}(\pi)$ is
already determined by $\delta(\pi)$.

Let $\Sigma^*$ denote the set of finite words, including the empty word $\epsilon$, over the alphabet $\Sigma$.
For a domino $\delta$, let
\[
 u(\delta)\in\BBB^*, \qquad \ell(\delta)\in\RRR^*
\]
be obtained by reading, respectively, the upper and lower cells from
bottom to top.

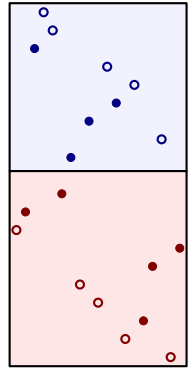
\begin{figure}[ht]
{\centering
\setplotptradius{3pt}
  \begin{tikzpicture}[xscale=0.12,yscale=0.24,line join=round]
    \fill[blue!5,draw=black,thick] (.25,.25) rectangle (19.75,9.5);
    \fill[red!10,draw=black,thick] (.25,.25) rectangle (19.75,-10.5);

    \foreach \x/\y in {3/7,7/1,9/3,12/4 }
        \node[permpt,blue!50!black] at (\x,\y) {};
    \foreach \x/\y in {4/9,5/8,11/6,14/5,17/2}
        \node[permpt,blue!50!black,thick,hollow] at (\x,\y) {};

    \foreach \x/\y in {2/2,6/1,15/8,16/5,19/4}
        \node[permpt,red!50!black] at (\x,-\y) {};
    \foreach \x/\y in {1/3,8/6,10/7,13/9,18/10}
        \node[permpt,red!50!black,thick,hollow] at (\x,-\y) {};
  \end{tikzpicture}
\par}
\caption{A domino, in which red left-to-right minima and blue right-to-left
maxima are drawn as hollow circles.}
\label{fig:domino}
\end{figure}

For the domino in Figure~\ref{fig:domino}, for example,
\begin{align*}
 u(\delta) &= \lb\lbo\lb\lb\lbo\lbo\lb\lbo\lbo,\\
 \ell(\delta) &= \lro\lro\lr\lro\lro\lr\lr\lro\lr\lr.
\end{align*}
Then $w_{\mathrm v}(\pi)$ is a \emph{shuffle} of $u(\delta(\pi))$ and $\ell(\delta(\pi))$: that is, $w_{\mathrm v}(\pi)$ is formed by interleaving the two words $u(\delta(\pi))$ and $\ell(\delta(\pi))$, whilst preserving the order within each of these two subwords.
The original binary word
$\chi(\pi)$ is recovered from $w_{\mathrm v}(\pi)$ by forgetting the
decorations.

The following observation, essentially due to
B\'ona~\cite{bona:a-new-upper-bou:}, is the basic restriction on these words.
A \emph{factor} of a word is a contiguous subword.

\begin{proposition}[{B\'ona~\cite[Lemmas~3.2 and~3.3]{bona:a-new-upper-bou:}}]
\label{prop:BR-free}
For every $\pi\in\av(1324)$, neither $w_{\mathrm p}(\pi)$ nor
$w_{\mathrm v}(\pi)$ contains the factor $\lb\lr$.
\end{proposition}

\begin{proof}
Suppose first that $w_{\mathrm v}(\pi)$ contains a factor $\lb\lr$,
corresponding to consecutive points $b<r$ in value order. Since $r$ is
red and lies above the blue point $b$, the staircase property forces $r$
to lie to the left of $b$. As $r$ is internal, there is a red point $a$
below and to the left of $r$; as $b$ is internal, there is a blue point
$d$ above and to the right of $b$. Since $b$ and $r$ are consecutive in
value order, we have $a<b<r<d$. Thus the four points occur from left to
right as
\[
 a,\ r,\ b,\ d,
\]
and form an occurrence of $1324$.

Now suppose that $w_{\mathrm p}(\pi)$ contains a factor $\lb\lr$,
corresponding to consecutive positions occupied by a blue point $b$
followed by a red point $r$. Since a red point cannot lie northeast of a
blue point, we have $r<b$. Again choose a red point $a$ below and to the
left of $r$, and a blue point $d$ above and to the right of $b$. The
consecutiveness of $b$ and $r$ in position order gives the left-to-right
order
\[
 a,\ b,\ r,\ d,
\]
while $a<r<b<d$. These four points again form a $1324$, a contradiction.
\end{proof}

Given $u\in\BBB^*$ and $\ell\in\RRR^*$, define
\[
 \SSS(u,\ell)
 =
 \left\{
 \text{$\lb\lr$-avoiding shuffles of $u$ and $\ell$}
 \right\},
 \qquad
 s(u,\ell)=|\SSS(u,\ell)|.
\]
The preceding discussion immediately gives the basic inequality on which
the rest of the paper is built.

\begin{proposition}
\label{prop-1324-leq-shuffles}
For every $n$,
\[
 |\av_n(1324)|
 \leq
 \sum_{\delta\in\DDD_n}s(u(\delta),\ell(\delta)).
\]
\end{proposition}

\begin{proof}
The map
\[
 \pi\longmapsto\bigl(\delta(\pi),w_{\mathrm v}(\pi)\bigr)
\]
is injective, since forgetting the decorations from $w_{\mathrm v}(\pi)$
recovers $\chi(\pi)$ and the earlier map
$\pi\mapsto(\delta(\pi),\chi(\pi))$ is injective. By
Proposition~\ref{prop:BR-free}, the second coordinate is a
$\lb\lr$-avoiding shuffle of $u(\delta(\pi))$ and
$\ell(\delta(\pi))$.
\end{proof}

%
%
%
%
%
%
%
%
%
%
%
%
\subsection{Overview of the method}\label{subsec:overview}

The precise words $u(\delta)$ and $\ell(\delta)$ contain more information than we can conveniently retain while summing over dominoes. Instead, we seek a bound depending on two simpler statistics $i(\delta)$ and $r(\delta)$ which we now introduce.

Define
\[
 i(\delta)=|u(\delta)|_{\lb}+|\ell(\delta)|_{\lr}
\]
where $|w|_\mathsf{X}$ counts the number of occurrences of letter $\mathsf{X}$ in the word $w$, and
\[
 r(\delta)=\mathop{runs}_\lb(u(\delta))+\mathop{runs}_\lr(\ell(\delta)),
\]
where $\mathop{runs}_\mathsf{X}(w)$ is the number of maximal nonempty runs of the letter $\mathsf{X}$
in $w$.  Thus, $i(\delta)$ records the number of internal points in the domino (each encoded either by $\lb$ or $\lr$), while $r(\delta)$ counts the sum of the number of internal runs in the two words $u(\delta)$ and $\ell(\delta)$.

The statistic $i$ is complementary to the leaf statistic used in~\cite[Section~6.1]{bbep:1324}, but $r$ does not seem to have appeared in the literature previously. However, there are some statistics in the literature with connections to one or both of the ones we use. The quantity $|u(\delta)|_{\lb}$ counts the internal points in the $\av(213)$ portion of a domino. If we denote the 213-avoiding permutation formed by all the points in this portion of the domino by $\sigma$, then in the inverse permutation $\sigma^{-1}$ (which also avoids 213), the internal points are precisely those that are at the start of an \emph{ascent}.  In other words, the statistic $i(\delta)$ is related to the ascent statistic on 132- and 213-avoiding permutations, which Simion~\cite{Simion1994} showed are counted by the Narayana numbers (OEIS sequence A001263). Similarly, the statistic $r(\delta)$ is related to counts of \emph{ascending runs} -- see Park and Park~\cite{park:avoiding-permutations} and Barnabei, Bonetti, Castronuovo and Silimbani~\cite{barnabei:ascending-runs}.

Our aim is to derive an estimate that takes the form
\[ 
	s(u(\delta),\ell(\delta)) \le C\lambda^{|\delta|} x^{i(\delta)}y^{r(\delta)}
\]
for some fixed positive $C$, $\lambda$, $x$ and $y$. Substituting this into the inequality of Proposition~\ref{prop-1324-leq-shuffles} gives
\[
|\av_n(1324)| \leq C\lambda^n \sum_{\delta\in\DDD_n} x^{i(\delta)}y^{r(\delta)}.
\]
Now define $\displaystyle D(z;x,y) = \sum_{\delta\in\DDD} z^{|\delta|}x^{i(\delta)}y^{r(\delta)}$. (Thus $D(z;1,1)$ is the ordinary generating function for dominoes.) Let $d_n(x,y)$ denote the terms involving dominoes of length $n$, i.e.
\[
d_n(x,y) = [z^n]D(z;x,y) = \sum_{\delta\in\DDD_n}x^{i(\delta)}y^{r(\delta)}.
\]
We thus have the inequality
\[
|\av_n(1324)| \leq C\lambda^n d_n(x,y),
\]
from which we see that $\gr(\av(1324)) \leq \lambda \limsup_{n\to\infty}\sqrt[n]{d_n(x,y)}$.

For this bound to be useful, we need to make informed choices for $\lambda$, $x$ and $y$, and we need to derive a form for the function $D(z;x,y)$, which counts dominoes weighted by their internal points and maximal runs of internal points.

%
%
%
%
%
%
%
%
%
%
%
%
\section{Weighted dominoes and arch systems}
\label{sec:weighted-dominoes}

In this section, we derive an expression for $D(z;x,y)$ that enables us to calculate, for any given $x$ and $y$, the quantity $\limsup_{n\to\infty}\sqrt[n]{d_n(x,y)}$. The derivation follows that used in~\cite{bbep:1324}.

\subsection{Arch systems}

We say that an $n$-point \emph{arch system} comprises $n$ points on a line, together with a (possibly empty) collection of noncrossing arcs, all on the same side of the line, with the property that no point is the left endpoint of more than one arc, and no point is the right endpoint of more than one arc. As observed in~\cite[Proposition 3.2]{bbep:1324}, arch systems on $n$ points are counted by the Catalan numbers, and are in bijection with both $\av_n(132)$ and $\av_n(213)$. We refer the reader to~\cite{bbep:1324} for the details, but here highlight the fact that each arc corresponds to an occurrence of 12 having the form $k(k+1)$ for some $k$. See Figure~\ref{fig:archsystems}, and note that we have drawn the arch system corresponding to the 213-avoider above the line, and that for the 132-avoider below the line.

\begin{figure}[ht]
{\centering
\setplotptradius{3pt}
  \begin{tikzpicture}[scale=0.2,line join=round]
    \fill[blue!5] (.25,.25) rectangle (9.75,9.75);
    \draw[thick] (.25,.25) rectangle (9.75,9.75);
    \foreach \x in {1,...,9} \draw [gray,very thin] (\x,-3.75)--(\x,9.75);
    \draw [gray,thin] (0,-4)--(10,-4);
    \plotperm[blue!50!black]{7,0,0,1,3,0,4,0,0}
    \plotperm[blue!50!black,thick,hollow]{0,9,8,0,0,6,0,5,2}
    \draw [white,very thick] (1,-4.25) arc [radius=2,start angle=180,end angle=360];  
    \draw [blue!50!black,very thick] (3,-3.75) arc [radius=1,start angle=0,end angle=180];
    \draw [blue!50!black,very thick] (7,-3.75) arc [radius=1,start angle=0,end angle=180];
    \draw [blue!50!black,very thick] (8,-3.75) arc [radius=0.5,start angle=0,end angle=180];
    \draw [blue!50!black,very thick] (9,-3.75) arc [radius=2.5,start angle=0,end angle=180];
    \plotperm[blue!50!black]{-4, 0, 0,-4,-4, 0,-4, 0, 0}
    \plotperm[blue!50!black,thick,hollow]{ 0,-4,-4, 0, 0,-4, 0,-4,-4}
    \begin{scope}[shift={(20,-1)}]
    \fill[red!10] (.25,.25) rectangle (10.75,10.75);
    \draw[thick] (.25,.25) rectangle (10.75,10.75);
    \foreach \x in {1,...,10} \draw [gray,very thin] (\x,-1.25)--(\x,10.75);
    \draw [gray,thin] (0,-1.5)--(11,-1.5);
    \plotperm[red!50!black]{ 0,9,10, 0,0,0,3,6,0,7}
    \plotperm[red!50!black,thick,hollow]{8, 0, 0,5,4,2,0,0,1,0}
    \draw [red!50!black,very thick] (1,-1.75) arc [radius=0.5,start angle=180,end angle=360];
    \draw [red!50!black,very thick] (2,-1.75) arc [radius=0.5,start angle=180,end angle=360];
    \draw [red!50!black,very thick] (4,-1.75) arc [radius=2,start angle=180,end angle=360];
    \draw [red!50!black,very thick] (6,-1.75) arc [radius=0.5,start angle=180,end angle=360];
    \draw [red!50!black,very thick] (8,-1.75) arc [radius=1,start angle=180,end angle=360];
    \plotperm[red!50!black]{   0,-1.5,-1.5,    0,   0,   0,-1.5,-1.5,   0,-1.5}
    \plotperm[red!50!black,thick,hollow]{-1.5,   0,   0,-1.5,-1.5,-1.5,   0,   0,-1.5,   0}
  \end{scope}
  \end{tikzpicture}\par}
\caption{
A 213-avoider (left) and a 132-avoider (right)
with their corresponding arch systems. The right-to-left maxima and left-to-right minima have been drawn as hollow circles.}\label{fig:archsystems}
\end{figure}
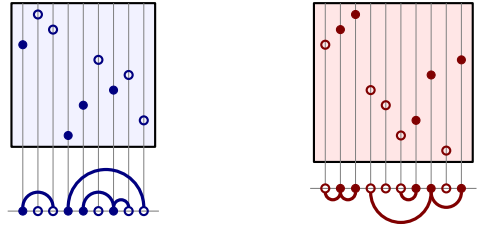

It is worth noting that in the bijection between $\av(213)$ and arch systems, the right-to-left maxima correspond precisely to the points of the arch system that are not left endpoints of any arc. Similarly, the left-to-right minima of a permutation in $\av(132)$ correspond to points that are not right endpoints.

Using these two bijections, a domino corresponds to an interleaving of two arch systems: one above (for the 213-avoiding permutation in the upper cell) and one below (for the 132-avoider in the lower cell). To distinguish whether a point with no incident arcs belongs to the upper or lower system, we draw small half-arcs, as shown in Figure~\ref{fig:domino-to-system}.  The condition that the domino avoids 1324 translates into a single condition on the resulting interleaved arch configurations, which is as follows.

\begin{proposition}[{Bevan et al.~\cite[Proposition 3.3]{bbep:1324}}]
  The set $\DDD$ of dominoes is in bijection with arch configurations that do not contain the pattern $\archavoid$.
\end{proposition}

\begin{figure}
[ht]
{\centering
\setplotptradius{3pt}
  \begin{tikzpicture}[xscale=0.12,yscale=0.24,line join=round]
    \fill[blue!5,draw=black,thick] (.25,.25) rectangle (19.75,9.5);
    \fill[red!10,draw=black,thick] (.25,.25) rectangle (19.75,-10.5);
    
    \foreach \x/\y in {3/7,7/1,9/3,12/4 }
    	\node[permpt,blue!50!black] at (\x,\y) {};
    \foreach \x/\y in {4/9, 5/8,11/6,14/5,17/2}
    	\node[permpt,blue!50!black,thick,hollow] at (\x,\y) {};
    
    \foreach \x/\y in {2/2,6/1,15/8,16/5,19/4}
    	\node[permpt,red!50!black] at (\x,-\y) {};
    \foreach \x/\y in {1/3,8/6,10/7,13/9,18/10}
    	\node[permpt,red!50!black,thick,hollow] at (\x,-\y) {};

    \begin{scope}[shift={(25,-1)},xscale=2]
    \bluePt{4}
    \bluePt{11}
    \redPt{10}
    \redPt{18}
    \foreach \s/\len in {1/.5,2/2,8/4,13/1,16/1.5} {\redArc{\s}{\len}}
    \foreach \s/\len in {17/5,14/1,12/1.5,5/1} {\blueArc{\s}{\len}}
    \draw [gray,thin] (0,1)--(20,1);
    \plotperm[blue!50!black]            {0,0,1,0,0,0,1,0,1,0,0,1,0,0,0,0,0,0,0}
    \plotperm[blue!50!black,hollow,thick] {0,0,0,1,1,0,0,0,0,0,1,0,0,1,0,0,1,0,0}
    \plotperm[red!50!black]                   {0,1,0,0,0,1,0,0,0,0,0,0,0,0,1,1,0,0,1}
    \plotperm[red!50!black,hollow,thick]{1,0,0,0,0,0,0,1,0,1,0,0,1,0,0,0,0,1,0}
    \end{scope}
  \end{tikzpicture}
\par}
\caption{A domino and its corresponding pair of arch systems. Note that left-to-right minima and right-to-left maxima have been drawn as hollow circles.}\label{fig:domino-to-system}
\end{figure}
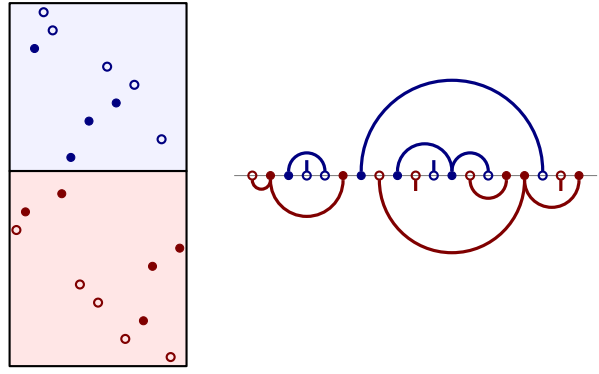

%
%
%
\subsection{Deriving and solving the functional equation}

Our remaining task in this section is to derive and solve the functional equation for $D(z;x,y)$, via the $\archavoid$-avoiding pairs of arch systems (and thus of dominoes). The method of~\cite{bbep:1324} works here with few changes -- there are no particular additional complications as a result of the two variables tracking the statistics $i(\delta)$ and $r(\delta)$.

We begin by interpreting the statistics $i(\delta)$ and $r(\delta)$ in the context of arch configurations.  
Since each arc joins consecutive values in increasing position order, a connected chain of $k\geq1$ arcs corresponds to $k+1$ consecutive increasing values. In the upper system, all but the last of these points are internal, while in the lower system all but the first are internal. This shows that, under the arch-system bijection, the points of the domino encoded by $\lb$ are exactly the left endpoints of upper arcs, while the points encoded by $\lr$ are the right endpoints of lower arcs.   Moreover, the nontrivial maximal connected chains correspond precisely to the maximal $\lb$-runs and maximal $\lr$-runs in the words $u(\delta)$ and $\ell(\delta)$. That is, the statistic $i(\delta)$ counts the number of left endpoints of upper arcs and right endpoints of lower arcs, while $r(\delta)$ counts the number of chains of upper or lower arcs.

As in~\cite{bbep:1324}, we build the $\archavoid$-avoiding arch configurations from left to right. An \emph{arch prefix}, is an initial portion of such a configuration, possibly with open arcs whose right-hand endpoints lie outside the prefix.  Let $\AAA$ be the set of arch prefixes with no open upper arcs, and let
\[
 A(v)=A(z,v;x,y)
\]
be their generating function, where $z$ marks points and $v$ marks open
lower arcs.  The variable $x$ marks left endpoints of upper arcs and right
endpoints of lower arcs, and $y$ marks the nontrivial maximal connected chains. We adopt the convention that we mark each nontrivial chain when it is completed: thus, if an arch prefix has one or more open lower arcs, these open arcs have not yet contributed to $y$.

Since a complete configuration has no open arcs, note that $D(z;x,y)=A(0)$.

\begin{proposition}
\label{prop:domino-functional-equation}
The series $A(v)$ satisfies
\begin{equation}
\label{eq:domino-functional-equation}
\begin{aligned}
 A(v)={}&1+2zA(v)+zvA(v)
 +xyz\frac{A(v)-A(0)}{v}\\
 &\quad+xz\bigl(A(v)-A(0)\bigr)
 +\frac{xyz^2A(v)^2}{1-xzA(v)}.
\end{aligned}
\end{equation}
\end{proposition}

\begin{proof}
Decompose a nonempty prefix according to its rightmost point, which may belong to either the upper or lower parts. If it is in the lower part, then it can be a lower
isolated point, a lower left endpoint, a lower right endpoint,  or a lower
point that is both a left and a right endpoint. These contributions are, respectively,
\[
 zA(v),\qquad zvA(v),\qquad
 xyz\frac{A(v)-A(0)}{v},\qquad
 xz\bigl(A(v)-A(0)\bigr).
\]
If it is in the upper part, then either it is an upper isolated point (with contribution $zA(v)$), or it terminates a connected sequence of $k\geq1$ upper arcs, as in the following diagram.

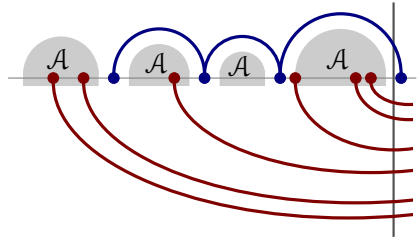
{\centering
\begin{tikzpicture}[scale=0.2,line join=round]
    \draw [white] (1,4.5)--(1,-10); 
    \fillBlueArc{6}{2.5}
    \fillBlueArc{12}{2}
    \fillBlueArc{17}{1.5}
    \fillBlueArc{25}{3}
    \draw [gray,thin] (0,1)--(27,1);
    \draw [gray!60!black,thick] (25.5,6)--(25.5,-9.5);
    \blueArc{13}{3}
    \blueArc{18}{2.5}
    \blueArc{26}{4}
    \node[] at (3.35,2.25) {$\AAA$};
    \node[] at (9.85,2) {$\AAA$};
    \node[] at (15.35,1.825) {$\AAA$};
    \node[] at (21.85,2.25) {$\AAA$};
    \begin{scope}
      \clip (0,-10) rectangle (26.75,6);
      \redArc{3}{19.5 and 9}
      \redArc{5}{18 and 8}
      \redArc{11}{13 and 6}
      \redArc{19}{6.5 and 4.5}
      \redArc{23}{3.5 and 2.5}
      \redArc{24}{2.5 and 1.5}
    \end{scope}
    \plotperm[blue!50!black]{0,0,0,0,0,0,1,0,0,0,0,0,1,0,0,0,0,1,0,0,0,0,0,0,0,1}
    \plotperm[red!50!black]{0,0,1,0,1,0,0,0,0,0,1,0,0,0,0,0,0,0,1,0,0,0,1,1,0,0}
  \end{tikzpicture}
\par}

This sequence contains $k$ left endpoints and forms a single chain of upper arcs, and hence its contribution is $yx^kz^{k+1}A(v)^{k+1}$. Summing over $k$, we obtain
\[
 y\sum_{k\geq1}x^kz^{k+1}A(v)^{k+1}
 =\frac{xyz^2A(v)^2}{1-xzA(v)}.
\]
Adding the empty prefix gives~\eqref{eq:domino-functional-equation}.
\end{proof}

To solve the functional equation, clear the denominators in
\eqref{eq:domino-functional-equation} and rearrange to obtain
\[
 P(A(v),A(0),z,v;x,y)=0,
\]
where
\begin{align*}
 P(X,Y,z,v;x,y)={}&xz\bigl(v^2z+vxz-vyz+2vz-v+xyz\bigr)X^2\\
 &+\bigl(v-v^2z-2vz -xyz-Yx^2z^2(v+y)\bigr)X\\
 &+Yxz(v+y)-v.
\end{align*}
We now use the method of iterated discriminants of Bousquet-M\'elou and Jehanne~\cite{bousquet-melou:polynomial-equations:}.
The iterated discriminant is
\[
Q(Y,z;x,y) = \operatorname{disc}_v\operatorname{disc}_X P= -256x^4y^4z^8\,R_1(Y,z;x,y)^2R_2(Y,z;x,y),
\]
where
\[
R_1(Y,z;x,y) = x^3z^3Y^2 +xz\bigl(1+(y-3x-2)z\bigr)Y +2(x+1)z-1
\]
and
\begin{align*}
R_2(Y,z;x,y)={}&x^4z^4\bigl(1+2(y-1)z\bigr)Y^3\\
&+x^2z^2\Bigl(2+(18y-4x-8)z\\
&\hspace{26mm}+(-x^2-6xy+8x+27y^2-36y+8)z^2\Bigr)Y^2\\
&+\Bigl(1-(4x+6)z+(2x^2-18xy+16x+12)z^2\\
&\hspace{26mm}+(4x^3-12x^2y-4x^2+36xy-16x-8)z^3\Bigr)Y\\
&-1+(4x+4)z+(-4x^2+16xy-8x-4)z^2
\end{align*}
is a cubic in $Y$. Now, the minimum polynomial for $Y=A(0)$ divides one of the factors of $Q(Y,z;x,y)$, since $Q(D,z;x,y)=0$. However
\[
R_1(D(z;x,y),z;x,y)=-1+O(z),
\]
does not vanish. Thus the minimum polynomial for $A(0)=D$ divides $R_2$, and we have established the following proposition.
\begin{proposition}
\label{prop:domino-cubic}
The series $D=D(z;x,y)$ satisfies
\begin{equation}
\label{eq:domino-cubic}
 R_2(D,z;x,y)=0,
\end{equation}
where $R_2$ is as defined above.
\end{proposition}

Note that at $x=y=1$, equation~\eqref{eq:domino-cubic} becomes
\[
 R_2(D,z;1,1) = z^4D^3+2z^2(3z+1)D^2+(12z^2-10z+1)D+8z-1=0,
\]
recovering the ordinary domino equation derived in~\cite{bbep:1324}. Also, for interest, we have
\begin{align*} D(z;x,y)={}&1+2z+(4+2xy)z^2\\ &+(8+12xy+2x^2y)z^3\\ &+(16+48xy+16x^2y+9x^2y^2+2x^3y)z^4+\cdots. \end{align*}

Finally for this section, let us record some facts about the asymptotics of $D$ that we will require later. For fixed positive $x$, $y$, let
\[\mu(x,y)=\limsup_{n\to\infty}d_n(x,y)^{1/n}.\]
Then $\mu(x,y)$ is the \emph{weighted growth rate} of the dominoes.

To compute $\mu$, we compute the discriminant of $R_2$ with respect to $D$. Since $R_2$ is a cubic in $D$, this is
\begin{equation}
\label{eq:domino-discriminant}
 \disc_D (R_2)
 =-x^4yz^5\Phi(z;x,y)^3,
\end{equation}
where
\begin{align*}
\Phi(z;x,y)={}&-4+\Big(12x-9y+24\Big)z\\
&+\Big(-12x^2+72xy-48x+36y-48\Big)z^2\\
&+\Big(4x^3-36x^2y+24x^2+108xy^2-144xy+48x-36y+32\Big)z^3.
\end{align*}
In particular, when $x=y=1$, we have $\Phi(z;1,1)=27z-4$.

\begin{theorem}\label{thm:weighted-domino-growth}
For every $x,y>0$, the polynomial $\Phi(z;x,y)$ has exactly one zero $\rho(x,y)$ in $(0,\tfrac12)$. This zero is the radius of convergence of $D(z;x,y)$, and hence
\[
\mu(x,y)=\frac1{\rho(x,y)}.
\]
\end{theorem}

\begin{proof}
Fix $x,y>0$, and write
\[
 d_n(x,y)=[z^n]D(z;x,y)
 =\sum_{\delta\in\DDD_n}x^{i(\delta)}y^{r(\delta)}.
\]
Let $\rho$ be the radius of convergence of $D(z;x,y)$.  There are
$2^n$ dominoes of size $n$ in which both cells are formed of decreasing sequences of points (since there are no restrictions on the interleavings available).  Every point of such a
domino is a left-to-right minimum or right-to-left maximum, so each has weight~$1$.  Moreover, since
$0\leq i(\delta),r(\delta)\leq n$ for any domino $\delta$, we have
\[
 2^n\leq d_n(x,y)\leq M^n d_n(1,1),
 \qquad \text{where }M=\max\{1,x\}\max\{1,y\}.
\]
The ordinary domino sequence has finite exponential growth, and hence
\[
 0<\rho\leq\frac12.
\]
Since the coefficients are nonnegative, Pringsheim's theorem
\cite[Theorem~IV.6]{flajolet:analytic-combin:} implies that $z=\rho$ is a singularity.

The finite singularities of an algebraic function satisfying a
polynomial equation are contained among the zeros of the leading
coefficient and the discriminant of that polynomial; see
\cite[Section~VII.7]{flajolet:analytic-combin:}.  The leading coefficient of
$R_2(D,z;x,y)$ in~\eqref{eq:domino-cubic}, viewed as a polynomial in $D$,
is
\[
 x^4z^4\bigl(1+2(y-1)z\bigr).
\]
For $0<z\leq\tfrac12$ this is nonzero, since
\[
 1+2(y-1)z=(1-2z)+2yz>0.
\]
Thus $\rho$ is a zero of the discriminant.  By equation~\eqref{eq:domino-discriminant}, we conclude that $\Phi(\rho;x,y)=0$.

It remains to show that $\Phi(z;x,y)$ has exactly one zero in
$(0,\tfrac12]$.  The substitution $z=u/(1+2u)$ maps $(0,\infty)$
bijectively onto $(0,\tfrac12)$.  Put
\[
\begin{aligned}
 Q(u)&=(1+2u)^3\Phi\!\left(\frac{u}{1+2u};x,y\right)\\
 &=4x(x^2-9xy+27y^2)u^3
   +12x(6y-x)u^2+3(4x-3y)u-4.
\end{aligned}
\]
With a view to applying Descartes' rule of signs, we consider the signs of the coefficients.
The leading coefficient of $Q$ is positive, since
\[
x^2-9xy+27y^2
=
\left(x-\frac92y\right)^2+\frac{27}{4}y^2>0.
\]

Now if $x\leq6y$, the coefficient of $u^2$ is non-negative, and thus (regardless of the sign of the coefficient of $u$) the coefficients of $Q$ have exactly one sign change.  Thus, when $x\leq 6y$ we conclude that $Q$ has exactly one positive root.

If $x>6y$, the coefficient of $u^2$ is negative, and that of $u$ is positive. Thus there are one or three positive roots in this case. We have
\[
 \disc_u Q
 =-314928\,xy^2(x-y)^3<0,
\]
so $Q$ has exactly one real root.  Since $Q(0)=-4$ and its leading
coefficient is positive, this root is positive.

Hence in either case $Q$ has exactly one positive root, and thus $\Phi(z;x,y)$ has exactly one zero in $(0,\tfrac12)$.  Furthermore,
\[
 \Phi\!\left(\frac12;x,y\right)
 =\frac{x}{2}(x^2-9xy+27y^2)>0,
\]
so $\tfrac12$ is not a zero.

It follows that $\rho$ is the unique zero of $\Phi(z;x,y)$ in
$(0,\tfrac12)$, and therefore its smallest positive zero.  Finally,
the Exponential Growth Formula \cite[Theorem~IV.7]{flajolet:analytic-combin:} gives
\[
 \mu(x,y)=\limsup_{n\to\infty}d_n(x,y)^{1/n}=\rho^{-1},
\]
as required.
\end{proof}

The function $\rho(x,y)$ defines a surface, whose structure we will largely leave unexplored. However, before we finish this section, we observe that when $x=y$, then we can factorise $\Phi$ as
\begin{align*}
\Phi(z;x,x) &= -4+(3x+24)z +(60x^2-12x-48)z^2 +(76x^3-120x^2+12x+32)z^3\\
&= \bigl(1+2(x-1)z\bigr)^2 \bigl((19x+8)z-4\bigr).
\end{align*}
Thus $\mu(x,x) = (19x+8)/4$.

%
%
%
%
%
%
%
%
%
%
%
%
\section{Shuffle automata}
\label{sec:shuffle-automata}

Fix a domino $\delta$, and let $u=u(\delta)$ and $\ell=\ell(\delta)$. In the previous section, we enumerated dominoes while recording their
internal points and internal runs.  In this section, we use these two statistics in order to bound the quantity $s(u,\ell)$, which counts the number of shuffles of $u$ and $\ell$ that avoid the factor $\lb\lr$. To do this we adopt some of the language of the \emph{shuffle automata} considered by Cunningham~\cite{cunningham:on-shuffling:}.

Informally, a shuffle automaton $\AAA$ is a state machine with two input tapes and one output tape. During each state transition, the next letter from one of the input tapes is copied to the output tape. The machine accepts when both input tapes are exhausted, with the resulting shuffled word stored on the output tape. In other words, every state of the automaton can be an accepting state, when both tapes have been exhausted. We note that a formal definition of such a machine (as given in~\cite{cunningham:on-shuffling:}) adopts the viewpoint of a monoid $\Sigma^*\times\Sigma^*\times\Sigma^*$, where $\Sigma$ is the underlying alphabet of the shuffled word (in our case, $\{\lbo,\lb,\lr,\lro\}$), but we do not require that here.

For our purposes, we will regard the two input tapes as being \emph{red} and \emph{blue}, according to whether the input word on the tape is over the alphabet $\RRR=\{\lr,\lro\}$ or $\BBB=\{\lbo,\lb\}$, respectively. The set of states of the automaton will be denoted $Q$, and the transitions are formed as a subset of $Q\times Q\times \{\lbo,\lb,\lr,\lro\}$, indicating the source state, target state, and letter being written during the transition, respectively. Note that it is clear which input tape the letter being written comes from, since the two alphabets for the input words are disjoint.

Typically speaking, the machines we describe are nondeterministic, in the sense that the machine has a choice at each step as to which input tape the letter being written came from. Indeed, as we are attempting to count $s(u,\ell)$, we are interested in the number of different outputs available when a given shuffle automaton is initialised with the input words $u$ and $\ell$. 

We are interested only in shuffles of $u$ and $\ell$ that avoid the factor $\lb\lr$. To record this, the states of any such shuffle automaton must keep track of whether the most recent letter written to the output was $\lb$ or not: if it was $\lb$, then the next letter output cannot be $\lr$. The simplest automaton for this situation is as follows.

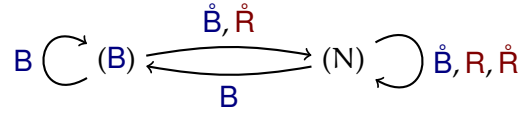
{\centering
\begin{tikzpicture}
\node[empty,align=center] (B) at (0,0) {($\lb$)};
\node[empty,align=center] (N) at (3,0) {($N$)};
\draw[thick] (B) edge[->,out=10,in=170] node[midway, above] {$\lbo,\lro$} (N);
\draw[thick] (N) edge[->,out=-170,in=-10] node[midway, below] {$\lb$} (B);
\draw[thick] (B) edge[->,loop,out=-150,in=150,looseness=5] node[midway,left] {$\lb$} (B);
\draw[thick] (N) edge[<-,loop,out=-30,in=30,looseness=5] node[midway,right] {$\lbo,\lr,\lro$} (N);
\end{tikzpicture}\par}

Here, we have labelled the states as ($\lb$), to indicate that the most recently written letter was $\lb$, and ($N$) to indicate the most recently written letter was not $\lb$. Note that there is no transition labelled $\lr$ from the state ($\lb$), as required to avoid $\lb\lr$-factors. To initialise this automaton, we start in state $N$ (since there are no output letters, and hence the previous output letter is certainly not $\lb$). The automaton accepts in either state, precisely when both input tapes are exhausted.

It is clear that this nondeterministic shuffle automaton can be used to generate all shuffles of two input words $u$ and $\ell$ that avoid $\lb\lr$. That is, it generates the whole of the set $\SSS(u,\ell)$. Furthermore, since the letters in the two input strings come from distinct alphabets, it follows that each output word is generated in exactly one way. Thus, the number of accepted runs of the above automaton is precisely $s(u,\ell) = |\SSS(u,\ell)|$.

\subsection{State weightings}
\label{sec:state-weightings}

We next give a simple device for counting the number of accepted runs of a shuffle automaton.  First, we need to set some notation concerning states and transitions of the shuffle automaton.

Let $\mathcal A$ be a shuffle automaton with (finite) state set $Q$.
For a run of $\AAA$ using fixed input words $u$ and $\ell$, an \emph{instantaneous description} (frequently abbreviated to ID) is a triple $(q;u',\ell')$ that records the state of $\AAA$ at a point during a run: the automaton is in state $q\in Q$, and $u'$ and $\ell'$ denote the unread suffixes of the input words $u$ and $\ell$.  An ID is \emph{reachable} if there exists some sequence of transitions from the initial state to the ID. 

Fix a reachable ID $(q;u',\ell')$. Let $s_q(u',\ell')$ denote the number of accepted runs beginning at this ID: thus $s_q(\epsilon,\epsilon)=1$, and if $q^*$ denotes the starting state of the automaton, then $s_{q^*}(u,\ell)$ is the total number of accepted runs generated by $\AAA$ with the input words $u$ and $\ell$.

Let $E(q;u',\ell')$ denote the set of transitions available from $(q;u',\ell')$. We call these the \emph{active} transitions for the reachable ID. Let $q_e$ denote the state of $Q$ reached by following the active transition $e\in E(q;u',\ell')$, and let $u_e'$ and $\ell_e'$ denote the suffixes in the two input tapes after this transition. (Note that exactly one of $u'=u_e'$ and $\ell'=\ell_e'$ holds, while the other suffix is now one letter shorter than before the transition.)  Thus, the possible IDs reachable in one step from $(q;u',\ell')$ are 
\[
\left\{ (q_e;u_e',\ell_e'): e \in E(q;u',\ell')\right\}.
\]
Thus, whenever $|u'|+|\ell'|>0$,
\begin{equation}
\label{eq:unweighted-shuffle-recurrence}
 s_q(u',\ell')
 =
 \sum_{e\in E(q;u',\ell')}
 s_{q_e}(u_e',\ell_e').
\end{equation}

The above gives us an iterative way to count the number of accepted runs on a given input. To make use of this, we seek to bound the quantities $s_q(u',\ell')$ in simpler terms. The mechanism we use for this is to apply weights to the states in $Q$.  Formally speaking, a \emph{state weighting} of $\mathcal A$ is a function
\[
 \kappa:Q\longrightarrow\mathbb R_{>0}.
\]

Fix $\lambda>0$.  We say that $\kappa$ is
\emph{$\lambda$-good} for $\mathcal A$ if
\begin{equation}
\label{eq:unweighted-certificate}
 \sum_{e\in E(q;u',\ell')}
 \kappa(q_e)
 \leq
 \lambda\kappa(q)
\end{equation}
for every pair of input words $u$ and $\ell$, and every reachable ID
$(q;u',\ell')$ with $|u'|+|\ell'|>0$.  It is worth noting that while the inequality must be satisfied at every reachable ID of the form $(q;u',\ell')$ (that is, an infinite number of inequalities), the number of distinct inequalities will always be finite.

\begin{proposition}
\label{prop:unweighted-certificate-bound}
Suppose that $\kappa$ is $\lambda$-good for a shuffle
automaton $\mathcal A$ with finite state set $Q$.  Then, for every reachable ID $(q;u',\ell')$,
\begin{equation}
\label{eq:unweighted-word-bound}
 s_q(u',\ell')
 \leq
 \frac{\kappa_+}{\kappa_-}
 \lambda^{|u'|+|\ell'|},
\end{equation}
where $\kappa_-=\min_{q\in Q}\kappa(q)$ and $\kappa_+=\max_{q\in Q}\kappa(q)$.
\end{proposition}

\begin{proof}
We claim, by induction on $|u'|+|\ell'|$, that the following inequality holds at any reachable ID $(q;u',\ell')$.
\[
 s_q(u',\ell')
 \leq
 \frac{\kappa(q)}{\kappa_-}
 \lambda^{|u'|+|\ell'|}.
\]

If $u'=\ell'=\epsilon$, then $s_q(\epsilon,\epsilon)=1$, and $1\leq\frac{\kappa(q)}{\kappa_-}$.

Otherwise, by~\eqref{eq:unweighted-shuffle-recurrence},~\eqref{eq:unweighted-certificate} and the induction
hypothesis,
\begin{align*}
 s_q(u',\ell') &=  \sum_{e\in E(q;u',\ell')}
 s_{q_e}(u_e',\ell_e')\\
 &\leq
 \sum_{e\in E(q;u',\ell')} \frac{\lambda^{|u'|+|\ell'|-1}}{\kappa_-}
\kappa(q_e)\\
 &\leq
 \frac{\kappa(q)}{\kappa_-}
 \lambda^{|u'|+|\ell'|}.
\end{align*}
The statement in the proposition now follows since $\kappa(q)\leq\kappa_+$.
\end{proof}

We now briefly return to the two-state automaton described above as an example.  From a reachable ID $(q;u',\ell')$, there is at most one active transition from each tape. Thus, if $\kappa: \{\lb,N\} \to \mathbb{R}_{>0}$ is a $\lambda$-good weighting, then the inequalities that must be satisfied are as follows.

{\centering
\begin{tabular}{c|c|c}
State&Active transition label(s)&Inequality\\\hline
$N$&$\{\lbo,\lro\}$ or $\{\lbo,\lr\}$& $\kappa(N) + \kappa(N) \leq \lambda \kappa(N)$\\
$N$&$\{\lb,\lro\}$ or $\{\lb,\lr\}$& $\kappa(\lb) + \kappa(N) \leq \lambda \kappa(N)$\\
$\lb$&$\{\lbo,\lro\}$&$\kappa(N) + \kappa(N) \leq \lambda \kappa(\lb)$\\
$\lb$&$\{\lb,\lro\}$&$\kappa(\lb) + \kappa(N) \leq \lambda \kappa(\lb)$\\
$\lb$&$\{\lbo\}$&$\kappa(N) \leq \lambda \kappa(\lb)$\\
$\lb$&$\{\lb\}$&$\kappa(\lb)  \leq \lambda \kappa(\lb)$\\
\end{tabular}\par}

Note that the final two rows (in which the next letter on the red tape is $\lr$) are automatically satisfied provided the inequality on the third-to-last row is satisfied. More significantly, the first row shows that $\lambda$-good weightings exist only if $\lambda\geq 2$. Indeed, it is straightforward to verify that the weighting $\kappa(\lb)=\kappa(N)=1$ is $2$-good.

It follows by Proposition~\ref{prop:unweighted-certificate-bound} that for any input words $u$ and $\ell$, we have
\[
 s(u,\ell)\leq2^{|u|+|\ell|}.
\]
Applying this to the two words associated with each $n$-point domino,
and recalling that $\gr(\DDD)=27/4$, we recover the upper bound from~\cite{bbep:1324}, namely
\[
 \gr(\av(1324))
 \leq
 2\cdot\frac{27}{4}
 =
 \frac{27}{2}.
\]
 It is worth remarking on the fact that this method gives no improvement on the naive upper bound from~\cite{bbep:1324}. This may come as a surprise, given the shuffles we consider here have the condition that they contain no $\lb\lr$-factor. This is because the estimate provided by Proposition~\ref{prop:unweighted-certificate-bound} cannot exploit the $\lb\lr$-factor avoidance: in the table of inequalities given above, the final two rows correspond to this limitation, but neither of these inequalities can ever be tight. 
 \subsection{Tracking internal points and runs}
\label{sec:tracking-internal-runs}

The two-state automaton records all the information required to enforce
the avoidance of the factor $\lb\lr$.  To make use of the weighted
enumeration of dominoes from the previous section, however, we require
a little more information. 

For words $u\in\BBB^*$ and $\ell\in\RRR^*$, write
\[
 i(u,\ell)=|u|_{\lb}+|\ell|_{\lr}
\]
for the total number of internal letters, and let $r(u,\ell)$ be the total number of runs of $\lb$ in $u$ and of $\lr$ in $\ell$.
For the two words $u(\delta)$ and $\ell(\delta)$ associated with a domino $\delta$, these are exactly
the statistics $i(\delta)$ and $r(\delta)$ defined in Subsection~\ref{subsec:overview}, and recorded by the generating function $D(z;x,y)$. Thus, to construct automata that can track $i(u,\ell)$ and $r(u,\ell)$, we need to know whether each internal letter that is read from the input tapes continues a run, or begins a new one.

For this purpose we refine the two-state automaton to one with
states of the form $(a,b,r)$, where $a \in\{\lb,N\}$, $b\in\BBB$ and $r\in\RRR$. The first coordinate is the same as the two-state automaton ($a=\lb$ if and only if the most recently written output letter was
$\lb$), while the second and third coordinates $b$ and $r$ record
the most recently written blue and red letters, respectively.

Not all eight triples $(a,b,r)$ are possible.  If $a=\lb$, then the most
recently-written blue letter must also be $\lb$, and hence the two
states of the form $(\lb,\lbo,r)$ are unreachable.  Let $\HHH$ denote the set comprising the other six states, thus
\[
 \HHH=
 \left\{
 (N,\lbo,\lro),(N,\lbo,\lr),
 (N,\lb,\lro),(N,\lb,\lr),
 (\lb,\lb,\lro),(\lb,\lb,\lr)
 \right\}.
\]
The transitions between states again correspond to the letter being copied from one of the input tapes to the output, and the resulting automaton is as follows.

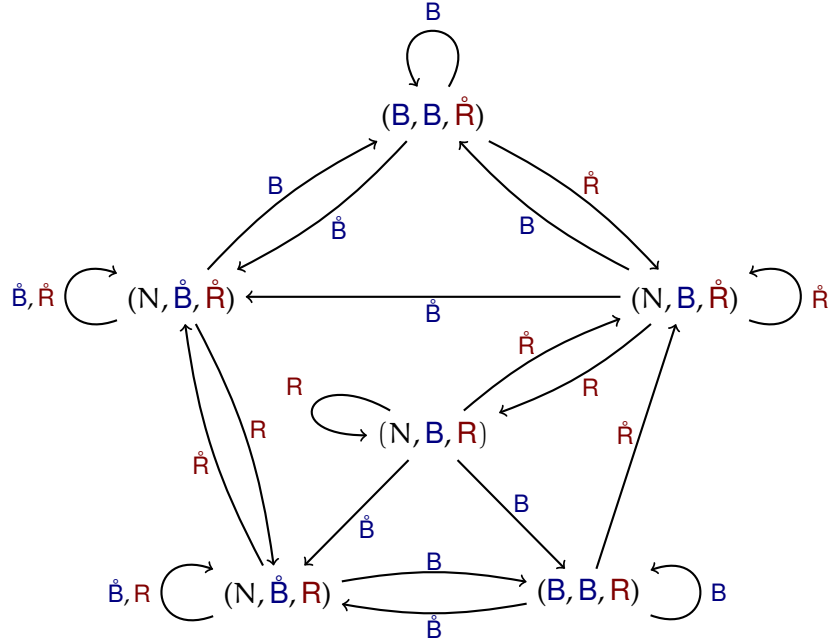
{\centering
\begin{tikzpicture}
\node (4) at (0,-.75) {$(N,\lb,\lr)$};
\foreach \n/\lab [count=\i] in {3/{N,\lb,\lro},5/{\lb,\lb,\lro},1/{N,\lbo,\lro},2/{N,\lbo,\lr},6/{\lb,\lb,\lr}}
	\node (\n) at (72*\i-54:3.5) {($\lab$)};
\foreach \s/\t\lab in {1/2/\lr,1/5/\lb,2/1/\lro,2/6/\lb,3/5/\lb,3/4/\lr,4/3/\lro,5/1/\lbo,5/3/\lro,6/2/\lbo}
	\draw[thick] (\s) edge [->,bend left=10] node[inner sep=0.5pt,midway,auto] {\footnotesize $\lab$} (\t); 
\foreach \s/\t\lab in {3/1/\lbo,4/2/\lbo,4/6/\lb,6/3/\lro}
	\draw[thick] (\s) edge [->] node[inner sep=0.5pt,midway,auto] {\footnotesize $\lab$} (\t); 
\draw[thick] (1) edge[<-,loop,out=160,in=200,looseness=4] node[midway,left] {\footnotesize $\lbo,\lro$} (1);
\draw[thick] (2) edge[<-,loop,out=160,in=200,looseness=4] node[midway,left] {\footnotesize $\lbo,\lr$} (2);
\draw[thick] (3) edge[<-,loop,out=20,in=-20,looseness=4] node[midway,right] {\footnotesize $\lro$} (3);
\draw[thick] (4) edge[<-,loop,out=180,in=150,looseness=7] node[midway,above left] {\footnotesize $\lr$} (4);
\draw[thick] (5) edge[<-,loop,out=120,in=60,looseness=7] node[midway,above] {\footnotesize $\lb$} (5);
\draw[thick] (6) edge[<-,loop,out=20,in=-20,looseness=4] node[midway,right] {\footnotesize $\lb$} (6);
\end{tikzpicture}\par}

The starting state is $(N,\lbo,\lro)$, so that the second and third coordinates do not contribute to any counts of internal runs.

Note that the two states labelled $(\lb,\lb,r)$ only have 3 transitions to other states (one of which is a loop to the same state), in order to avoid creating the forbidden factor $\lb\lr$.  As before, an input tape can be
read whenever the corresponding transition (whose label equals the first unread letter on the tape) is available, and a run is
accepted when both input tapes have been exhausted. The accepted runs of this automaton with input words $u$ and $\ell$ are again in bijection with the elements of $\SSS(u,\ell)$. 

This automaton allows us to identify when internal runs are being continued and when new internal runs are being created. For example, when the letter $\lb$ is copied from the blue tape to the output at state $(a,b,r)$, it begins a new
internal run if and only if $b=\lbo$. Similarly, when $\lr$ is read from the red
tape, a new internal run begins precisely when $r=\lro$.

\subsection{Weighted transitions}
\label{sec:weighted-transitions}

In order to track the two statistics $i(u,\ell)$ and $r(u,\ell)$, we need to assign weights to the transitions of a shuffle automaton.  
Let $\AAA$ be a shuffle automaton with state set $Q$. A \emph{transition weighting} of $\AAA$ is a function $\tau$ from the set of transitions of $\AAA$ to $\mathbb{R}_{>0}$.
The \emph{weight} of a run is the product of the weights of its transitions, with the empty run having weight~$1$.

For a reachable ID $(q;u',\ell')$, let $F_q(u',\ell')$ denote the total weight of all accepted runs beginning at this ID.
Thus $F_q(\epsilon,\epsilon)=1$, and, whenever $|u'|+|\ell'|>0$,
\begin{equation}
\label{eq:weighted-shuffle-recurrence}
 F_q(u',\ell')
 =
 \sum_{e\in E(q;u',\ell')}
 \tau(e)F_{q_e}(u_e',\ell_e').
\end{equation}

Fix $\lambda>0$. Given an automaton with transition weighting $\tau$,  we say that a state weighting
$\kappa:Q\to\mathbb R_{>0}$ is \emph{$\lambda$-good} if
\begin{equation}
\label{eq:weighted-good}
 \sum_{e\in E(q;u',\ell')}
 \tau(e)\kappa(q_e)
 \leq
 \lambda\kappa(q)
\end{equation}
for every pair of input words $u$ and $\ell$, and for every reachable ID $(q;u',\ell')$ with $|u'|+|\ell'|>0$.  
Note that the earlier definition of $\lambda$-good for unweighted automata is recovered by taking $\tau(e)=1$ for every transition. 
We have the corresponding weighted version of Proposition~\ref{prop:unweighted-certificate-bound}. 

\begin{proposition}
\label{prop:weighted-certificate-bound}
Suppose that $\kappa$ is $\lambda$-good for a shuffle automaton $\AAA$ with finite state set $Q$ and transition weighting $\tau$.  Then, for every reachable ID $(q;u',\ell')$,
\[
 F_q(u',\ell')
 \leq
 \frac{\kappa_+}{\kappa_-}
 \lambda^{|u'|+|\ell'|},
\]
where $\kappa_-=\min_{q\in Q}\kappa(q)$ and $\kappa_+=\max_{q\in Q}\kappa(q)$.
\end{proposition}

\begin{proof}
The proof is the same induction as in
Proposition~\ref{prop:unweighted-certificate-bound}.  If both input
tapes are exhausted, then $F_q(\epsilon,\epsilon)=1$, which is at most
$\kappa(q)/\kappa_-$.  Otherwise,
\eqref{eq:weighted-shuffle-recurrence}, \eqref{eq:weighted-good} and the induction hypothesis give
\begin{align*}
 F_q(u',\ell')
 &\leq
 \frac{\lambda^{|u'|+|\ell'|-1}}{\kappa_-}
 \sum_{e\in E(q;u',\ell')}
 \tau(e)\kappa(q_e)\\
 &\leq
 \frac{\kappa(q)}{\kappa_-}
 \lambda^{|u'|+|\ell'|}.
\end{align*}
The result follows since $\kappa(q)\leq\kappa_+$.
\end{proof}

It follows from Proposition~\ref{prop:weighted-certificate-bound} that the total weight of the accepted runs on any pair of input words having total length $n$ is at most $C\lambda^n$, where $C=\frac{\kappa_+}{\kappa_-}$.

Our choice of weighting is dictated by the form of the bound that we seek for the method described in Subsection~\ref{subsec:overview}, namely
\[ 
	s(u(\delta),\ell(\delta)) \le C\lambda^{|\delta|} x^{i(\delta)}y^{r(\delta)}.
\]  
To fit the form in the equation above, we need -- via Proposition~\ref{prop:weighted-certificate-bound} -- the quantity $F_{q^*}(u,\ell)$ (where $q^*$ denotes the starting state of the automaton) to equal $s(u,\ell)x^{-i(u,\ell)}y^{-r(u,\ell)}$. Thus, we require that each accepted run on input words $(u,\ell)$ has weight $x^{-i(u,\ell)}y^{-r(u,\ell)}$, for fixed $x,y>0$. We obtain the following.

\begin{corollary}
\label{cor:weighted-shuffle-bound}
Fix $x,y>0$, and let $\AAA$ be a weighted shuffle automaton such that, for every pair $(u,\ell)$, its accepted runs are in bijection with the admissible shuffles of $u$ and $\ell$, and every such run has weight $x^{-i(u,\ell)}y^{-r(u,\ell)}$.

If $\AAA$ admits a $\lambda$-good state weighting $\kappa$, then
\[
 s(u,\ell)
 \leq
 \frac{\kappa_+}{\kappa_-}
 \lambda^{|u|+|\ell|}
 x^{i(u,\ell)}y^{r(u,\ell)}
\]
for every $u\in\BBB^*$ and $\ell\in\RRR^*$.
\end{corollary}
\begin{proof}
The total weight of the accepted runs on $(u,\ell)$ is $s(u,\ell)x^{-i(u,\ell)}y^{-r(u,\ell)}$.
By Proposition~\ref{prop:weighted-certificate-bound}, this is at most $\frac{\kappa_+}{\kappa_-}\lambda^{|u|+|\ell|}$.
Rearranging gives the result.
\end{proof}
The analysis in Subsection~\ref{subsec:overview} now gives us the following general result. Recall that 
\[
 d_n(x,y) = [z^n]D(z;x,y) = \sum_{\delta\in\DDD_n}  x^{i(\delta)}y^{r(\delta)} ,
\]
and for any positive $x$ and $y$,
\[
 \mu(x,y) = \limsup_{n\to\infty}d_n(x,y)^{1/n}.
\]
Thus, the quantity $\mu(x,y)$ is the reciprocal of the radius of convergence of $D(z;x,y)$.

\begin{proposition}
\label{prop:weighted-transfer}
Fix $x,y>0$.  If there exists a weighted shuffle automaton satisfying
the hypotheses of Corollary~\ref{cor:weighted-shuffle-bound} and
admitting a $\lambda$-good state weighting, then
\[
 \gr(\av(1324))
 \leq
 \lambda\mu(x,y).
\]
\end{proposition}

\begin{proof}
For each $\delta\in\DDD_n$,
Corollary~\ref{cor:weighted-shuffle-bound} gives
\[
 s\bigl(u(\delta),\ell(\delta)\bigr)
 \leq
 \frac{\kappa_+}{\kappa_-}
 \lambda^n x^{i(\delta)}y^{r(\delta)}.
\]
Summing over all $n$-point dominoes and using
Proposition~\ref{prop-1324-leq-shuffles}, we obtain
\[
 |\av_n(1324)|
 \leq
 \frac{\kappa_+}{\kappa_-}\lambda^n d_n(x,y).
\]
The result follows on taking $n$th roots and the lim sup as $n\to\infty$.
\end{proof}


We conclude this section by briefly revisiting the six-state automaton from Section~\ref{sec:tracking-internal-runs}. A transition weighting that satisfies the hypotheses of Corollary~\ref{cor:weighted-shuffle-bound} is as follows. For a transition $e$ from $(a,b,r)$, set
\[
\tau(e) = \begin{cases}
1&\text{if $e$ reads $\lbo$ or $\lro$.}\\
x^{-1}&\text{if $e$ reads $\lb$ and $b=\lb$; or $e$ reads $\lr$ and $r=\lr$}\\
(xy)^{-1}&\text{otherwise.}
\end{cases}
\]
This weighting ensures that every internal letter contributes a factor $x^{-1}$ to the weight of the run,
and the first letter of each internal run contributes an additional factor $y^{-1}$. It is straightforward to verify that this weighting ensures every accepted run with input $(u,\ell)$ has the required weight $x^{-i(u,\ell)}y^{-r(u,\ell)}$. 

Numerical optimisation of this six-state system gives no improvement on the existing upper bound for $\gr(\av(1324))$: the optimum found is at
\[
x=y=1,\qquad \lambda=2,
\]
which returns the bound $13.5$.

The six-state automaton accepts exactly the same $\lb\lr$-avoiding shuffles as the two-state automaton; the additional information recorded by the extra states serves only to distinguish whether an internal letter begins or
continues a run, and hence to implement the weights above.  
To obtain a genuine improvement, we next build larger automata to use information about future letters on the two input tapes.

\section{Lookahead automata}
\label{sec:lookahead}

In this section we combine the three coordinates $(a,b,r)$ from the six-state automaton introduced earlier with a \emph{lookahead}: that is, the automaton stores a fixed number of unread letters from the blue and red tapes.  A $k$-lookahead automaton has state set
\[
 Q_k= \HHH \times\BBB^k\times\RRR^k.
\]
(Recall that $\HHH$ denotes the six states from the earlier automaton.) Thus $|Q_k|=6\cdot 2^{2k}$. We write its states as
\[
 (a,b,r;U,L),
\]
where $U=u_1\cdots u_k$ records the next $k$ blue letters and $L=\ell_1\cdots \ell_k$ records the next $k$ red letters. (If fewer than $k$ unread letters remain on a tape, the lookahead is padded with $\lbo$ on the blue tape and $\lro$ on the red tape.) 
Throughout, a letter is regarded as \emph{unread} until it is copied to the output. In other words, the `lookahead' facility provided by the words $U$ and $L$ merely provides information as to what the next \(k\) letters of each suffix are.

As with the six-state automaton, each transition in the $k$-lookahead automaton represents a letter being copied from one of the input tapes to the output tape. At the same time, the word $U$ or $L$ (if the state at the start of the transition is $(a,b,r;U,L)$, say) must be updated to remove the first letter from the appropriate input tape, and include the next.  Thus, if the current blue letter $u_1$ is written to the output, then the word $U$ is updated to include the next symbol (say $u_{k+1}$) from the input tape, and removing $u_1$, i.e. we have the transition
\begin{equation}
\label{eq:k-blue-transition}
 (a,b,r;u_1u_2\cdots u_k,L)
 \longmapsto
 (a',u_1,r;u_2\cdots u_ku_{k+1},L),
\end{equation}
where $a'=\lb$ if and only if $u_1=\lb$. We use the term \emph{reveal} for the process of adding the next letter $u_{k+1}$ into the lookahead state, to distinguish from the reading/writing function of the automaton.

Similarly, writing the current red letter $\ell_1$ corresponds to the transition
\begin{equation}
\label{eq:k-red-transition}
 (a,b,r;U,\ell_1\ell_2\cdots \ell_k)
 \longmapsto
 (N,b,\ell_1;U,\ell_2\cdots \ell_k\ell_{k+1}),
\end{equation}
unless $a=\lb$ and $\ell_1=\lr$, in which case the first red letter cannot be read, to avoid forming a $\lb\lr$-factor. It is worth noting that the state $(a,b,r;U,L)$ has two transitions labelled $u_1$, according to the two possible values of the newly-revealed letter $u_{k+1}$. At any particular reachable ID, exactly one of these two transitions is active. Similarly (providing $a=N$ or $\ell_1=\lro$) there are two transitions labelled by $\ell_1$, with again exactly one active at any reachable ID. Thus, at any reachable ID the only choice is whether to write the next red or the next blue letter. 

Once either input tape is exhausted, no transition associated with that tape is active, so padding symbols are never copied to the output. 
A run is accepted precisely when both input tapes are exhausted. 

The initial state has $a=N$, $b=\lbo$ and $r=\lro$, and $U$ and $L$ contain the first $k$ letters of the two input tapes (padded by $\lbo$ and $\lro$, as required). 

For an edge $e$ emanating from a state $(a,b,r;U,L)$, the weighting $\tau$ follows that of the six-state automaton, namely 
\[
\tau(e) = \begin{cases}
1&\text{if $e$ writes $\lbo$ or $\lro$.}\\
x^{-1}&\text{if $e$ writes $\lb$ and $b=\lb$; or $e$ writes $\lr$ and $r=\lr$}\\
(xy)^{-1}&\text{otherwise.}
\end{cases}
\]

For fixed $k$, the process to check whether a given state weighting $\kappa$ is $\lambda$-good reduces to checking equation~\eqref{eq:weighted-good} for every state in $Q_k$ and
every possible pair in $\BBB\times\RRR$. Note that at the ends of the tapes, one or both transitions are simply omitted, and these cases do not require separate checking. Thus there are at most $4|Q_k|$ inequalities to check.

For fixed input words, `forgetting' the lookahead gives a bijection between accepted runs of this automaton and of the six-state automaton. Corresponding runs have the same transition weights, so the hypotheses of Corollary~\ref{cor:weighted-shuffle-bound} remain satisfied.

\subsection{One-step lookahead}
\label{sec:one-step-lookahead}

As a demonstration of the method, and our first improvement on the upper bound for $\gr(\av(1324))$, we consider the $1$-lookahead automaton.
  

The one-step automaton has states
\[
 Q_1= \HHH \times\BBB\times\RRR.
\]
We write these in the form $(a,b,r;u,\ell)$, where $(a,b,r)\in\mathcal H$, and $u,\ell$ are the first letters on the blue and red input tapes that have yet to be written to the output.  Thus $Q_1$ comprises $6\cdot2\cdot2=24$ states.

Suppose that the next letters on the blue and red tapes after $u$ and $\ell$ are
$u'\in\{\lbo,\lb\}$ and $\ell'\in\{\lro,\lr\}$, respectively. (We will refer to these as the \emph{second letters}.) For completeness, we describe the transitions explicitly for this automaton. Writing the current blue letter gives the transition
\begin{equation}
\label{eq:one-step-blue-transition}
 (a,b,r;u,\ell)\longmapsto(a',u,r;u',\ell),
\end{equation}
where $a'=N$ if and only if $u=\lbo$. If $a=N$ or $\ell=\lro$, then writing the current red letter gives
\begin{equation}
\label{eq:one-step-red-transition}
 (a,b,r;u,\ell)\longmapsto(N,b,\ell;u,\ell').
\end{equation}
If $a\neq N$ and $\ell\neq\lro$ there is no red transition, to avoid creating a $\lb\lr$-factor. The transition weights $\tau$ are as described in the general case.

We now identify a state weighting $\kappa$. This family was found by ChatGPT by first solving the finite system of $\lambda$-good inequalities numerically, identifying the inequalities which are tight at the optimum, and then solving those equalities symbolically. Remarkably, the resulting weighting can be defined using a single parameter $x$.

For any $x \in [1/\sqrt2,1]$, let
\begin{equation}
\label{eq:one-step-y-lambda}
 y=\frac{2x}{x+1},
 \qquad
 \lambda=1+\frac{1}{x}.
\end{equation}
Note that the choice of $\lambda=1+\frac1x$ cannot be made smaller, because of the inequality $\lambda\ge 1+1/x$ which arises from the state $q=(N,\lbo,\lr; \lbo,\lr)$: this state has two loops, corresponding to the cases when the first two remaining letters on the two input tapes are $\lbo\lbo$ and $\lr\lr$. The transition that writes $\lbo$ has weight 1, while the transition $\lr$ has weight $1/x \ge 1$. Thus, for any state weighting $\kappa$, the inequality associated with any reachable ID at $q$ whose remaining input suffixes begin $\lbo\lbo$ and $\lr\lr$ is
\[
\left(1+\frac1x\right)\kappa(q) \leq \lambda\kappa(q).
\]
This should be compared with the 2-state automaton example shown earlier, in which we established that $\lambda=2$. In a sense, the increase in value of $\lambda$ in this example is the `cost' for selecting a value of $x$ that is less than 1. 

For convenience, define the following indicator functions
\[
 \epsilon_{\lb}(b)=
 \begin{cases}
  0,&b=\lbo,\\
  1,&b=\lb,
 \end{cases}
 \qquad
 \epsilon_{\lr}(r)=
 \begin{cases}
  0,&r=\lro,\\
  1,&r=\lr.
 \end{cases}
\]
Also put
\[
 d_0=\frac{2x+1}{2x(x+1)},
 \qquad
 d_1=\frac{1}{2(x+1)}.
\]

Define a positive weighting $\kappa$ of each state $(a,b,r;u,\ell)$ by
\begin{equation}
\label{eq:one-step-state-weighting}
\begin{array}{c|cc}
 (u,\ell)&a=N&a=\lb\\ \hline
 (\lbo,\lro)
   &y&y\\[1mm]
 (\lb,\lro)
   &y^{\epsilon_{\lb}(b)}
   &y^{\epsilon_{\lb}(b)}\\[1mm]
 (\lbo,\lr)
   &y^{\epsilon_{\lr}(r)}
   &\dfrac12y^{\epsilon_{\lr}(r)+1}\\[3mm]
 (\lb,\lr)
   &d_0y^{\epsilon_{\lb}(b)+\epsilon_{\lr}(r)}
   &d_1y^{\epsilon_{\lb}(b)+\epsilon_{\lr}(r)}.
\end{array}
\end{equation}
Note that whenever $a=\lb$, then $b=\lb$, so $\epsilon_{\lb}(b)=1$ wherever it appears in the $a=\lb$ column.

\begin{proposition}
\label{prop:one-step-certificate}
For $x\in[1/\sqrt2,1]$ and $y=2x/(x+1)$, the state weighting in
\eqref{eq:one-step-state-weighting} is $\left((x+1)/x\right)$-good for the one-step
shuffle automaton.
\end{proposition}

\begin{proof}
For each of the $24$ states and each of the four possible pairs
$(u',\ell')$ of second letters, substitute
\eqref{eq:one-step-blue-transition} and
\eqref{eq:one-step-red-transition} into the inequality~\eqref{eq:weighted-good}. This produces 96 inequalities, detailed in Appendix~\ref{app:one-step-inequalities}.
After dividing by the weight of the current state, these inequalities reduce to the nonnegativity of the following eight quantities:
\[
 0,\qquad
 1,\qquad
 \frac{1-x}{x},\qquad
 \frac{1-x}{2x},\qquad
 \frac{2x^2-1}{2x(x+1)},\qquad
 \frac{2x^2-1}{2x^2},\qquad
 \frac{x}{2x+1},\qquad
 \frac{2x}{2x+1}.
\]
Note that these are all nonnegative for $x\in[1/\sqrt2,1]$.
At the end of either tape the corresponding transition is omitted, so
the same inequalities remain valid.
\end{proof}

It follows from Proposition~\ref{prop:weighted-transfer} that
\begin{equation}
\label{eq:one-step-family-bound}
 \gr(\av(1324))
 \leq
 \frac{x+1}{x}\,
 \mu\left(x,\frac{2x}{x+1}\right)
\end{equation}
for every $x\in[1/\sqrt2,1]$.

We now identify the optimal value of $x$. For any fixed $x$, by Theorem~\ref{thm:weighted-domino-growth}, $\mu\left(x,\frac{2x}{x+1}\right)$ is equal to the reciprocal of the unique zero of $\Phi\left(z;x,\frac{2x}{x+1}\right)$ in the interval $(0,\frac12)$. Write $\rho(x)$ for this zero, so the right-hand side of
\eqref{eq:one-step-family-bound} is
\[
 f(x)=\frac{x+1}{x\rho(x)}.
\]

Substituting $y=2x/(x+1)$ into $\Phi$ yields an expression
\[ 
\Phi\left(z;x,\frac{2x}{x+1}\right) = \frac{2}{(x+1)^2}P(x,z),
\]
where $P(x,z)$ is a polynomial, which necessarily satisfies $P(x,\rho(x))=0$.

Since $y=2x/(x+1)\neq x$ for $1/\sqrt2<x<1$, the discriminant calculation in the proof of Theorem~\ref{thm:weighted-domino-growth} shows that the relevant zero $\rho(x)$ is simple, and hence implicit differentiation is valid. 
Differentiating this equation implicitly, and imposing the stationary condition $f'(x)=0$, gives a second polynomial relation between $x$ and $\rho(x)$.  
Eliminating $\rho(x)$ from these two equations shows
that every stationary point in the interior of
$[1/\sqrt2,1]$ is a zero of
\[
 x^7-x^6-6x^5+6x^4+12x^3+12x^2-8x-8.
\]
This polynomial is strictly increasing on $[1/\sqrt2,1]$ and takes opposite signs at the endpoints, so it has a unique zero in this interval, namely $x_0\approx0.803998$. At this value,
\[
 y_0=\frac{2x_0}{x_0+1}
 \approx 0.891351,
\qquad\text{and}\qquad
 \lambda_0=\frac{x_0+1}{x_0}
           \approx2.243784.
\]
The corresponding radius of convergence and weighted domino growth are $\rho(x_0)\approx 0.167564$ and $\mu(x_0,y_0)\approx5.967858$.

Furthermore, since $x_0$ is the only possible stationary point in the interior of
the interval, it remains only to compare the endpoints.  We have
\[
 f(x_0)=13.3905851029\ldots,\qquad
 f(1/\sqrt2)=13.4291650750\ldots,\qquad
 f(1)=\frac{27}{2}=13.5.
\]
Thus $f$ attains its minimum on $[1/\sqrt2,1]$ at $x=x_0$.

\begin{theorem}
\label{thm:one-step-bound}
The growth rate of the $1324$-avoiding permutations satisfies
\[
 \gr(\av(1324))
 \leq
f(x_0)=13.3905851029\dots
\leq
 13.390586.
\]
\end{theorem}


Note that the unrounded bound $f(x_0)$ in Theorem~\ref{thm:one-step-bound} is an algebraic number: it is a
real root of
\[
 4z^7-20z^6-343z^5-1392z^4-1076z^3
 +7512z^2+12600z+6272=0.
\]

%
%
%
%
%
%
%
%
%
%
%
%

\subsection{Three-step lookahead}
\label{sec:three-step-lookahead}

For our final upper bound, we consider the three-step lookahead
automaton. Its state set is
\[
 Q_3=\HHH\times\BBB^3\times\RRR^3,
\]
and hence contains $6\cdot8^2=384$ states. The transition rules and transition weights are exactly as described in the general case.

For each state, the four possible pairs $(u_4,\ell_4)\in\BBB\times\RRR$
of newly revealed letters determine the local inequalities, so there
are $384\cdot4=1536$ inequalities to check.

Guided again by numerical optimisation conducted by ChatGPT, we choose the following parameters. Let $t$ be the unique zero\footnote{A direct calculation gives
\[
m(1158423/10^6)<0<m(1158424/10^6),
\]
while $m'(T)>0$ throughout this interval; hence the zero of $m$ in this interval is unique.} in
$(1158423/10^6,1158424/10^6)$ of
\begin{equation}
\label{eq:three-step-t}
 m(T)=3T^5+10T^4-T^3-11T^2-6T-1,
\end{equation}
and put
\begin{equation}
\label{eq:three-step-parameters}
 x=y=t^{-1},
 \qquad
 \lambda=1+t.
\end{equation}
Numerically, $t\approx1.1584233$, $x=y\approx0.863242$, and
$\lambda\approx2.158423$. The transition weights are therefore
$1$, $t$, and $t^2$, depending on whether the letter written is external,
continues an internal run, or begins an internal run, respectively. As in the
one-step case, states with two loops force
$\lambda\geq1+1/x=1+t$.

The complete list of state weights (found by ChatGPT using numerical optimisation, followed by algebraic reconstruction) is given in Appendix~\ref{app:three-step-state-weighting}. 
We express each state weighting using polynomials $K_q(T)\in\mathbb Q[T]$ of degree at most four,\footnote{The degree bound here is a consequence of the fact that $t$ satisfies $m(t)=0$ where $m$ is a quintic polynomial.} and then $\kappa(q) = K_q(t)$. Only 139 distinct polynomials occur among the 384 weights.

\begin{proposition}
\label{prop:three-step-certificate}
For the parameters in~\eqref{eq:three-step-parameters}, the state
weighting $\kappa$ given in
Appendix~\ref{app:three-step-state-weighting} is $(1+t)$-good for
the three-step shuffle automaton.
\end{proposition}

\begin{proof}

For each state $q$, let $K_q(T)\in\mathbb Q[T]$ be the
polynomial of degree at most four specified in
Appendix~\ref{app:three-step-state-weighting}, and define
\[
 \kappa(q)=K_q(t).
\]
For a state $q$ and a pair $(u_4,\ell_4)$ of newly-revealed
letters, form the polynomial
\[
 S_{q,u_4,\ell_4}(T) = (1+T)K_q(T) - \sum_e T^{d(e)}K_{q_e}(T),
\]
where $d(e)\in\{0,1,2\}$ is such that the corresponding transition weight is $1,t,t^2$, respectively. 
The evaluation of $S_{q,u_4,\ell_4}$ at $t$ is the \emph{slack} obtained by rearranging the inequality~\eqref{eq:weighted-good} for any reachable ID at state $q$ whose newly-revealed letters are $u_4$ and $\ell_4$. We must show that $S_{q,u_4,\ell_4}(t)\geq 0$.

Let $R_{q,u_4,\ell_4}(T)$ be the remainder in $\mathbb{Q}[T]$ when $S_{q,u_4,\ell_4}(T)$ is divided by $m(T)$. Thus $R_{q,u_4,\ell_4}(T)$ has degree at most 4. Since $m(t)=0$,
\[
 S_{q,u_4,\ell_4}(t)
 =
 R_{q,u_4,\ell_4}(t),
\]
and so we must verify that all 1536 numbers $R_{q,u_4,\ell_4}(t)$ are non-negative. To avoid issues using the irrational number $t$, we consider the values of the remainders across a small interval containing $t$ whose endpoints are rational. Put
$a=1158423/10^6$ and $b=1158424/10^6$. For each nonzero remainder
$R_{q,u_4,\ell_4}(T)=\sum_{j=0}^4r_jT^j$, we have
\[
 R_{q,u_4,\ell_4}(t)\geq
 \sum_{r_j\geq0}r_j a^j+
 \sum_{r_j<0}r_j b^j,
\]
and so it suffices to show that
$ \sum_{r_j\geq0}r_j a^j+
 \sum_{r_j<0}r_j b^j \geq 0$.

All of the 1536 remainders are nonnegative. (In fact, 514 are the zero polynomial, and 1022 are strictly positive, but we do not require this information for the proof.) 
When one of the input tapes no longer has any symbols, the corresponding transition from the ID is omitted, and the inequalities remain valid. Hence the weighting $\kappa$ is $\lambda$-good, when $\lambda=1+t$.

To show that the weighting $\kappa$ is itself strictly positive on each state, we use the same approach as taken for evaluating the slacks. Note that since the state weightings use only 139 polynomials, it suffices to verify that each of these polynomial is positive at $T=t$. This approach gives a strictly positive lower bound in every case. Hence \(\kappa(q)>0\) for all \(q\).
\end{proof}

That the 384 state weights are positive and the 1536 local inequalities hold
has been formally verified in Lean~4 using \texttt{mathlib}. The formalisation is available in the accompanying
repository: \url{https://github.com/rbrignall/1324-lookahead/}

The formalisation is divided into three main files.  The file \texttt{Arithmetic.lean} contains the elementary arithmetic used in the certificate, such as the evaluation of polynomials of degree $\leq 4$, reduction
using $m(t)=0$, and the lower-bound argument used on the interval $(a,b)$. The file \texttt{Data.lean} contains the same content as Appendix~\ref{app:three-step-state-weighting}, namely the 139 polynomials $w_i(T)$ and the six $8\times8$ lookup tables that assign them to the 384 states.  
Finally, \texttt{Certificate.lean} defines the states and transitions of the 3-lookahead automaton, constructs the slacks from these, and proves positivity of the state weights together with all
1536 $\lambda$-good inequalities. The main theorem is \texttt{three\_step\_certificate}. For further details, consult the README file in the repository.

We can now state and prove the main result of the paper.

\begin{theorem}
\label{thm:three-step-bound}
With $t$ as above, we have
\[
 \gr(\av(1324))
 \leq
 \frac{(1+t)(19+8t)}{4t}
 =13.1672475688056\ldots.
\]
\end{theorem}

\begin{proof}
In the case that $x=y$, we have, by the formula established after Theorem~\ref{thm:weighted-domino-growth},
\[
 \mu(t^{-1},t^{-1})=\frac{19+8t}{4t}.
\]
Thus Propositions~\ref{prop:weighted-transfer}
and~\ref{prop:three-step-certificate} imply
\[
 \gr(\av(1324))
 \leq(1+t)\mu(t^{-1},t^{-1})
 =\frac{(1+t)(19+8t)}{4t}.
\]
\end{proof}

%
%
%
%
%
%
\section{Concluding remarks}
\label{sec:conclusion}

We finish this paper with a few observations about potential steps to improve the upper bound further, including some (unverified) results from preliminary explorations.

\paragraph{$k$-lookahead for $k>3$} The obvious next step is to increase the lookahead in the automaton. However, the improvements are modest, while the computations get progressively worse: an unverified calculation by ChatGPT is that the 5-lookahead automaton gives an upper bound of approximately 13.1498, at a cost of 6144 states and 24,576 inequalities to check.

\paragraph{Avoiding other configurations} The appearance of $\lb\lr$-factors in a shuffle is of course not the only way to find a copy of 1324. One relatively simple modification one can make is to expand the alphabet to $\{\lbo,\lb_1,\lb_2,\lro,\lr_1,\lr_2\}$, so that $\lbo$ and $\lro$ denote extremal points as before, while $\lb_1$ encodes internal points for which there is exactly one right-to-left maximum in the region northeast of the point. Similarly, $\lr_1$ encodes internal points with exactly one left-to-right minimum to the southwest. Other internal points are encoded by $\lb_2$ and $\lr_2$. In this way, we obtain further forbidden words such as $\lb_2 x\lr_2$ where $x$ is any single letter. 

This can be encoded, and an additional statistic added to the domino generating function $D$. The resulting 3-lookahead automaton gives an (unverified) upper bound on $\gr(\av(1324))$ of $13.143$. However, because of the increase in alphabet size and the additional statistic that must be tracked, this automaton has over 13,000 states, and a collection of over 118,000 inequalities to check. Thus, again, the gains are rather modest given the extra effort required.

However, in principle, the notion of restricting the allowed shuffles (and finding a good way to bound the number of these succinctly) would appear to offer a reasonably promising route towards further improvements of the upper bound.

\bibliographystyle{plainurl}
\bibliography{refs}

\appendix

%
%
%
%
%
%
\clearpage
\section{Inequalities for the 1-lookahead automaton}

\label{app:one-step-inequalities}

\begingroup
For nonempty tapes, the inequalities arising from reachable IDs are determined by a state $q=(a,b,r;u,\ell)$ and second letters
$(u',\ell')$, and take the form
\mbox{$\lambda-\kappa(q)^{-1}\sum_e\tau(e)\kappa(q_e)\geq0$}. Note that
$y=2x/(x+1)$ and $\lambda=1+1/x$, for $1/\sqrt2\leq x\leq1$.
\endgroup

\par
\begingroup
\centering
\fontsize{8.5}{10.0}\selectfont
\setlength{\tabcolsep}{2.0pt}
\setlength{\aboverulesep}{0.7pt}
\setlength{\belowrulesep}{1.6pt}
\renewcommand{\arraystretch}{0.80}
\begin{tabular}{@{}lcl@{\hspace{1.4em}}lcl@{}}
\toprule
State $q$ & $(u',\ell')$ & Inequality &
State $q$ & $(u',\ell')$ & Inequality\\
\midrule
$(N,\lbo,\lro;\lbo,\lro)$ & $(\lbo,\lro)$ & $(1-x)/x\geq0$ & $(N,\lb,\lr;\lbo,\lro)$ & $(\lbo,\lro)$ & $(1-x)/x\geq0$\\
 & $(\lbo,\lr)$ & $(1-x)/(2x)\geq0$ &  & $(\lbo,\lr)$ & $(1-x)/(2x)\geq0$\\
 & $(\lb,\lro)$ & $(1-x)/(2x)\geq0$ &  & $(\lb,\lro)$ & $(1-x)/(2x)\geq0$\\
 & $(\lb,\lr)$ & $0\geq0$ &  & $(\lb,\lr)$ & $0\geq0$\\
\cmidrule(lr){1-3}\cmidrule(lr){4-6}
$(N,\lbo,\lro;\lbo,\lr)$ & $(\lbo,\lro)$ & $0\geq0$ & $(N,\lb,\lr;\lbo,\lr)$ & $(\lbo,\lro)$ & $0\geq0$\\
 & $(\lbo,\lr)$ & $0\geq0$ &  & $(\lbo,\lr)$ & $0\geq0$\\
 & $(\lb,\lro)$ & $(2x^2-1)/(2x(x+1))\geq0$ &  & $(\lb,\lro)$ & $(2x^2-1)/(2x(x+1))\geq0$\\
 & $(\lb,\lr)$ & $(2x^2-1)/(2x(x+1))\geq0$ &  & $(\lb,\lr)$ & $(2x^2-1)/(2x(x+1))\geq0$\\
\cmidrule(lr){1-3}\cmidrule(lr){4-6}
$(N,\lbo,\lro;\lb,\lro)$ & $(\lbo,\lro)$ & $0\geq0$ & $(N,\lb,\lr;\lb,\lro)$ & $(\lbo,\lro)$ & $0\geq0$\\
 & $(\lbo,\lr)$ & $(2x^2-1)/(2x(x+1))\geq0$ &  & $(\lbo,\lr)$ & $(2x^2-1)/(2x(x+1))\geq0$\\
 & $(\lb,\lro)$ & $0\geq0$ &  & $(\lb,\lro)$ & $0\geq0$\\
 & $(\lb,\lr)$ & $(2x^2-1)/(2x(x+1))\geq0$ &  & $(\lb,\lr)$ & $(2x^2-1)/(2x(x+1))\geq0$\\
\cmidrule(lr){1-3}\cmidrule(lr){4-6}
$(N,\lbo,\lro;\lb,\lr)$ & $(\lbo,\lro)$ & $0\geq0$ & $(N,\lb,\lr;\lb,\lr)$ & $(\lbo,\lro)$ & $0\geq0$\\
 & $(\lbo,\lr)$ & $x/(2x+1)\geq0$ &  & $(\lbo,\lr)$ & $x/(2x+1)\geq0$\\
 & $(\lb,\lro)$ & $x/(2x+1)\geq0$ &  & $(\lb,\lro)$ & $x/(2x+1)\geq0$\\
 & $(\lb,\lr)$ & $2x/(2x+1)\geq0$ &  & $(\lb,\lr)$ & $2x/(2x+1)\geq0$\\
\cmidrule(lr){1-3}\cmidrule(lr){4-6}
$(N,\lbo,\lr;\lbo,\lro)$ & $(\lbo,\lro)$ & $(1-x)/x\geq0$ & $(\lb,\lb,\lro;\lbo,\lro)$ & $(\lbo,\lro)$ & $(1-x)/x\geq0$\\
 & $(\lbo,\lr)$ & $(1-x)/(2x)\geq0$ &  & $(\lbo,\lr)$ & $(1-x)/(2x)\geq0$\\
 & $(\lb,\lro)$ & $(1-x)/(2x)\geq0$ &  & $(\lb,\lro)$ & $(1-x)/(2x)\geq0$\\
 & $(\lb,\lr)$ & $0\geq0$ &  & $(\lb,\lr)$ & $0\geq0$\\
\cmidrule(lr){1-3}\cmidrule(lr){4-6}
$(N,\lbo,\lr;\lbo,\lr)$ & $(\lbo,\lro)$ & $0\geq0$ & $(\lb,\lb,\lro;\lbo,\lr)$ & $(\lbo,\lro)$ & $0\geq0$\\
 & $(\lbo,\lr)$ & $0\geq0$ &  & $(\lbo,\lr)$ & $0\geq0$\\
 & $(\lb,\lro)$ & $(2x^2-1)/(2x(x+1))\geq0$ &  & $(\lb,\lro)$ & $(2x^2-1)/(2x^2)\geq0$\\
 & $(\lb,\lr)$ & $(2x^2-1)/(2x(x+1))\geq0$ &  & $(\lb,\lr)$ & $(2x^2-1)/(2x^2)\geq0$\\
\cmidrule(lr){1-3}\cmidrule(lr){4-6}
$(N,\lbo,\lr;\lb,\lro)$ & $(\lbo,\lro)$ & $0\geq0$ & $(\lb,\lb,\lro;\lb,\lro)$ & $(\lbo,\lro)$ & $0\geq0$\\
 & $(\lbo,\lr)$ & $(2x^2-1)/(2x(x+1))\geq0$ &  & $(\lbo,\lr)$ & $(2x^2-1)/(2x(x+1))\geq0$\\
 & $(\lb,\lro)$ & $0\geq0$ &  & $(\lb,\lro)$ & $0\geq0$\\
 & $(\lb,\lr)$ & $(2x^2-1)/(2x(x+1))\geq0$ &  & $(\lb,\lr)$ & $(2x^2-1)/(2x(x+1))\geq0$\\
\cmidrule(lr){1-3}\cmidrule(lr){4-6}
$(N,\lbo,\lr;\lb,\lr)$ & $(\lbo,\lro)$ & $0\geq0$ & $(\lb,\lb,\lro;\lb,\lr)$ & $(\lbo,\lro)$ & $0\geq0$\\
 & $(\lbo,\lr)$ & $x/(2x+1)\geq0$ &  & $(\lbo,\lr)$ & $0\geq0$\\
 & $(\lb,\lro)$ & $x/(2x+1)\geq0$ &  & $(\lb,\lro)$ & $1\geq0$\\
 & $(\lb,\lr)$ & $2x/(2x+1)\geq0$ &  & $(\lb,\lr)$ & $1\geq0$\\
\cmidrule(lr){1-3}\cmidrule(lr){4-6}
$(N,\lb,\lro;\lbo,\lro)$ & $(\lbo,\lro)$ & $(1-x)/x\geq0$ & $(\lb,\lb,\lr;\lbo,\lro)$ & $(\lbo,\lro)$ & $(1-x)/x\geq0$\\
 & $(\lbo,\lr)$ & $(1-x)/(2x)\geq0$ &  & $(\lbo,\lr)$ & $(1-x)/(2x)\geq0$\\
 & $(\lb,\lro)$ & $(1-x)/(2x)\geq0$ &  & $(\lb,\lro)$ & $(1-x)/(2x)\geq0$\\
 & $(\lb,\lr)$ & $0\geq0$ &  & $(\lb,\lr)$ & $0\geq0$\\
\cmidrule(lr){1-3}\cmidrule(lr){4-6}
$(N,\lb,\lro;\lbo,\lr)$ & $(\lbo,\lro)$ & $0\geq0$ & $(\lb,\lb,\lr;\lbo,\lr)$ & $(\lbo,\lro)$ & $0\geq0$\\
 & $(\lbo,\lr)$ & $0\geq0$ &  & $(\lbo,\lr)$ & $0\geq0$\\
 & $(\lb,\lro)$ & $(2x^2-1)/(2x(x+1))\geq0$ &  & $(\lb,\lro)$ & $(2x^2-1)/(2x^2)\geq0$\\
 & $(\lb,\lr)$ & $(2x^2-1)/(2x(x+1))\geq0$ &  & $(\lb,\lr)$ & $(2x^2-1)/(2x^2)\geq0$\\
\cmidrule(lr){1-3}\cmidrule(lr){4-6}
$(N,\lb,\lro;\lb,\lro)$ & $(\lbo,\lro)$ & $0\geq0$ & $(\lb,\lb,\lr;\lb,\lro)$ & $(\lbo,\lro)$ & $0\geq0$\\
 & $(\lbo,\lr)$ & $(2x^2-1)/(2x(x+1))\geq0$ &  & $(\lbo,\lr)$ & $(2x^2-1)/(2x(x+1))\geq0$\\
 & $(\lb,\lro)$ & $0\geq0$ &  & $(\lb,\lro)$ & $0\geq0$\\
 & $(\lb,\lr)$ & $(2x^2-1)/(2x(x+1))\geq0$ &  & $(\lb,\lr)$ & $(2x^2-1)/(2x(x+1))\geq0$\\
\cmidrule(lr){1-3}\cmidrule(lr){4-6}
$(N,\lb,\lro;\lb,\lr)$ & $(\lbo,\lro)$ & $0\geq0$ & $(\lb,\lb,\lr;\lb,\lr)$ & $(\lbo,\lro)$ & $0\geq0$\\
 & $(\lbo,\lr)$ & $x/(2x+1)\geq0$ &  & $(\lbo,\lr)$ & $0\geq0$\\
 & $(\lb,\lro)$ & $x/(2x+1)\geq0$ &  & $(\lb,\lro)$ & $1\geq0$\\
 & $(\lb,\lr)$ & $2x/(2x+1)\geq0$ &  & $(\lb,\lr)$ & $1\geq0$\\

\bottomrule
\end{tabular}\par
\endgroup

%
%
%
%
%
%
\section{Weightings for the 3-lookahead automaton}

\label{app:three-step-state-weighting}

\begingroup
Let $t$ be the unique zero in $(1158423/10^6,1158424/10^6)$ of
\[
 m(T)=3T^5+10T^4-T^3-11T^2-6T-1.
\]
The six tables below specify the 384 state weights. An entry $s$ in row $i$ and column $j$ of the table for $h\in\HHH$ means $\kappa(h;U_i,L_j)=w_s(t)$, where $w_s$ is one of 139 polynomials of the form
$w_s(T)=c_{s,0}+c_{s,1}T+c_{s,2}T^2+c_{s,3}T^3+c_{s,4}T^4$ described on subsequent pages.

The words $U_i,L_i$ have the binary index $i\in\{0,\ldots,7\}$:
the first letter is the leftmost bit, with $0$ for a circled
letter and $1$ for an internal letter. Thus, for the $U_i$,
\[ U_0 = \lbo\lbo\lbo, U_1=\lbo\lbo\lb, U_2=\lbo\lb\lbo, U_3 = \lbo\lb\lb, U_4 = \lb\lbo\lbo, U_5=\lb\lbo\lb, U_6=\lb\lb\lbo, U_7=\lb\lb\lb.\]
\endgroup

\begingroup
\fontsize{9}{11}\selectfont
\setlength{\tabcolsep}{2pt}
\renewcommand{\arraystretch}{1.03}

\noindent
\begin{minipage}[t]{.485\linewidth}
\vspace{0pt}
\centering
{\small\rule[-3pt]{0pt}{15pt} $h=(N,\lbo,\lro)$}\par\smallskip
\begin{tabular*}{\linewidth}{@{\extracolsep{\fill}}r|rrrrrrrr@{}}
\toprule
$i\backslash j$ & 0 & 1 & 2 & 3 & 4 & 5 & 6 & 7\\
\midrule
0 & 1 & 2 & 3 & 3 & 4 & 5 & 5 & 5\\
1 & 2 & 6 & 3 & 7 & 8 & 9 & 10 & 11\\
2 & 3 & 3 & 12 & 13 & 14 & 15 & 16 & 17\\
3 & 3 & 18 & 19 & 20 & 21 & 22 & 23 & 24\\
4 & 4 & 8 & 14 & 25 & 26 & 27 & 28 & 29\\
5 & 5 & 9 & 15 & 30 & 27 & 31 & 32 & 33\\
6 & 5 & 34 & 35 & 36 & 37 & 38 & 39 & 40\\
7 & 5 & 41 & 42 & 43 & 44 & 45 & 40 & 46\\
\bottomrule
\end{tabular*}
\end{minipage}
\hfill%
\begin{minipage}[t]{.485\linewidth}
\vspace{0pt}
\centering
{\small\rule[-3pt]{0pt}{15pt} $h=(N,\lbo,\lr)$}\par\smallskip
\begin{tabular*}{\linewidth}{@{\extracolsep{\fill}}r|rrrrrrrr@{}}
\toprule
$i\backslash j$ & 0 & 1 & 2 & 3 & 4 & 5 & 6 & 7\\
\midrule
0 & 1 & 2 & 3 & 3 & 47 & 3 & 3 & 3\\
1 & 2 & 6 & 3 & 7 & 48 & 49 & 50 & 51\\
2 & 3 & 3 & 12 & 13 & 52 & 53 & 54 & 55\\
3 & 3 & 18 & 19 & 20 & 56 & 57 & 58 & 59\\
4 & 4 & 8 & 14 & 25 & 60 & 61 & 62 & 63\\
5 & 5 & 9 & 15 & 30 & 61 & 64 & 65 & 66\\
6 & 5 & 34 & 35 & 36 & 67 & 68 & 69 & 70\\
7 & 5 & 41 & 42 & 43 & 71 & 72 & 70 & 73\\
\bottomrule
\end{tabular*}
\end{minipage}
\par\medskip
\noindent
\begin{minipage}[t]{.485\linewidth}
\vspace{0pt}
\centering
{\small\rule[-3pt]{0pt}{15pt} $h=(N,\lb,\lro)$}\par\smallskip
\begin{tabular*}{\linewidth}{@{\extracolsep{\fill}}r|rrrrrrrr@{}}
\toprule
$i\backslash j$ & 0 & 1 & 2 & 3 & 4 & 5 & 6 & 7\\
\midrule
0 & 1 & 2 & 3 & 3 & 4 & 5 & 5 & 5\\
1 & 2 & 6 & 3 & 7 & 8 & 9 & 10 & 11\\
2 & 3 & 3 & 12 & 13 & 14 & 15 & 16 & 17\\
3 & 3 & 18 & 19 & 20 & 21 & 22 & 23 & 24\\
4 & 47 & 48 & 52 & 74 & 60 & 61 & 62 & 63\\
5 & 3 & 49 & 53 & 75 & 61 & 64 & 65 & 66\\
6 & 3 & 76 & 77 & 78 & 67 & 68 & 69 & 70\\
7 & 3 & 79 & 80 & 81 & 71 & 72 & 70 & 73\\
\bottomrule
\end{tabular*}
\end{minipage}
\hfill%
\begin{minipage}[t]{.485\linewidth}
\vspace{0pt}
\centering
{\small\rule[-3pt]{0pt}{15pt} $h=(N,\lb,\lr)$}\par\smallskip
\begin{tabular*}{\linewidth}{@{\extracolsep{\fill}}r|rrrrrrrr@{}}
\toprule
$i\backslash j$ & 0 & 1 & 2 & 3 & 4 & 5 & 6 & 7\\
\midrule
0 & 1 & 2 & 3 & 3 & 47 & 3 & 3 & 3\\
1 & 2 & 6 & 3 & 7 & 48 & 49 & 50 & 51\\
2 & 3 & 3 & 12 & 13 & 52 & 53 & 54 & 55\\
3 & 3 & 18 & 19 & 20 & 56 & 57 & 58 & 59\\
4 & 47 & 48 & 52 & 74 & 82 & 83 & 84 & 85\\
5 & 3 & 49 & 53 & 75 & 83 & 86 & 87 & 88\\
6 & 3 & 76 & 77 & 78 & 89 & 90 & 91 & 92\\
7 & 3 & 79 & 80 & 81 & 93 & 94 & 92 & 95\\
\bottomrule
\end{tabular*}
\end{minipage}
\par\medskip
\noindent
\begin{minipage}[t]{.485\linewidth}
\vspace{0pt}
\centering
{\small\rule[-3pt]{0pt}{15pt} $h=(\lb,\lb,\lro)$}\par\smallskip
\begin{tabular*}{\linewidth}{@{\extracolsep{\fill}}r|rrrrrrrr@{}}
\toprule
$i\backslash j$ & 0 & 1 & 2 & 3 & 4 & 5 & 6 & 7\\
\midrule
0 & 1 & 2 & 3 & 3 & 96 & 97 & 97 & 97\\
1 & 2 & 6 & 3 & 7 & 98 & 99 & 100 & 101\\
2 & 3 & 3 & 12 & 13 & 102 & 102 & 103 & 104\\
3 & 3 & 18 & 19 & 20 & 105 & 106 & 107 & 108\\
4 & 47 & 48 & 52 & 74 & 109 & 110 & 110 & 110\\
5 & 3 & 49 & 53 & 75 & 111 & 111 & 112 & 113\\
6 & 3 & 76 & 77 & 78 & 114 & 115 & 115 & 115\\
7 & 3 & 79 & 80 & 81 & 116 & 117 & 117 & 117\\
\bottomrule
\end{tabular*}
\end{minipage}
\hfill%
\begin{minipage}[t]{.485\linewidth}
\vspace{0pt}
\centering
{\small\rule[-3pt]{0pt}{15pt} $h=(\lb,\lb,\lr)$}\par\smallskip
\begin{tabular*}{\linewidth}{@{\extracolsep{\fill}}r|rrrrrrrr@{}}
\toprule
$i\backslash j$ & 0 & 1 & 2 & 3 & 4 & 5 & 6 & 7\\
\midrule
0 & 1 & 2 & 3 & 3 & 118 & 119 & 119 & 119\\
1 & 2 & 6 & 3 & 7 & 120 & 121 & 122 & 123\\
2 & 3 & 3 & 12 & 13 & 124 & 124 & 125 & 126\\
3 & 3 & 18 & 19 & 20 & 127 & 128 & 129 & 130\\
4 & 47 & 48 & 52 & 74 & 131 & 132 & 132 & 132\\
5 & 3 & 49 & 53 & 75 & 133 & 133 & 134 & 135\\
6 & 3 & 76 & 77 & 78 & 136 & 137 & 137 & 137\\
7 & 3 & 79 & 80 & 81 & 138 & 139 & 139 & 139\\
\bottomrule
\end{tabular*}
\end{minipage}
\par\endgroup
\clearpage
\subsection*{Polynomial coefficients, $w_{1}$--$w_{70}$}
\begingroup\small\setlength{\parskip}{4pt}
Each row specifies $w_s(T)=c_{s,0}+c_{s,1}T+c_{s,2}T^2+c_{s,3}T^3+c_{s,4}T^4$.
\par\medskip\endgroup
\begingroup
\fontsize{9}{11}\selectfont
\setlength{\tabcolsep}{2pt}
\renewcommand{\arraystretch}{1.02}
\noindent
\begin{minipage}[t]{.49\linewidth}
\vspace{0pt}
\centering
\begin{tabular*}{\linewidth}{@{\extracolsep{\fill}}rrrrrr@{}}
\toprule
$s$ & $c_{s,0}$ & $c_{s,1}$ & $c_{s,2}$ & $c_{s,3}$ & $c_{s,4}$\\
\midrule
1 & $1848$ & $7728$ & $9360$ & $-10344$ & $-4032$\\
2 & $252$ & $756$ & $1152$ & $-1164$ & $-468$\\
3 & $24$ & $0$ & $0$ & $0$ & $0$\\
4 & $1400$ & $6488$ & $6684$ & $-7972$ & $-3036$\\
5 & $0$ & $24$ & $0$ & $0$ & $0$\\\hline
\addlinespace[2pt]
6 & $72$ & $72$ & $192$ & $-168$ & $-72$\\
7 & $-6054$ & $-27396$ & $-26232$ & $32706$ & $12312$\\
8 & $256$ & $1180$ & $1272$ & $-1472$ & $-564$\\
9 & $16$ & $48$ & $64$ & $-56$ & $-24$\\
10 & $-686$ & $-3106$ & $-3020$ & $3746$ & $1410$\\\hline
\addlinespace[2pt]
11 & $3726$ & $16848$ & $16092$ & $-20070$ & $-7560$\\
12 & $8$ & $-24$ & $-64$ & $56$ & $24$\\
13 & $-1264$ & $-5720$ & $-5464$ & $6832$ & $2568$\\
14 & $44$ & $188$ & $192$ & $-220$ & $-84$\\
15 & $8$ & $16$ & $0$ & $0$ & $0$\\\hline
\addlinespace[2pt]
16 & $-208$ & $-940$ & $-916$ & $1140$ & $432$\\
17 & $736$ & $3320$ & $3156$ & $-3948$ & $-1476$\\
18 & $-38271$ & $-172959$ & $-164850$ & $206049$ & $77499$\\
19 & $-8256$ & $-37324$ & $-35572$ & $44468$ & $16728$\\
20 & $29425$ & $132985$ & $126672$ & $-158369$ & $-59559$\\\hline
\addlinespace[2pt]
21 & $-1336$ & $-6046$ & $-5748$ & $7208$ & $2706$\\
22 & $-16989$ & $-76731$ & $-73042$ & $91387$ & $34347$\\
23 & $5966$ & $26978$ & $25716$ & $-32122$ & $-12090$\\
24 & $-18633$ & $-84171$ & $-80130$ & $100239$ & $37683$\\
25 & $-200$ & $-908$ & $-852$ & $1084$ & $408$\\\hline
\addlinespace[2pt]
26 & $640$ & $2920$ & $2832$ & $-3488$ & $-1320$\\
27 & $16$ & $64$ & $48$ & $-56$ & $-24$\\
28 & $-64$ & $-292$ & $-288$ & $368$ & $132$\\
29 & $100$ & $448$ & $420$ & $-512$ & $-204$\\
30 & $-2786$ & $-12584$ & $-11984$ & $14982$ & $5652$\\\hline
\addlinespace[2pt]
31 & $24$ & $104$ & $104$ & $-112$ & $-48$\\
32 & $-326$ & $-1466$ & $-1380$ & $1762$ & $642$\\
33 & $1074$ & $4854$ & $4644$ & $-5742$ & $-2214$\\
34 & $-4182$ & $-18908$ & $-18074$ & $22564$ & $8490$\\
35 & $-1344$ & $-6078$ & $-5812$ & $7264$ & $2730$\\
\bottomrule
\end{tabular*}
\end{minipage}
\hfill%
\begin{minipage}[t]{.49\linewidth}
\vspace{0pt}
\centering
\begin{tabular*}{\linewidth}{@{\extracolsep{\fill}}rrrrrr@{}}
\toprule
$s$ & $c_{s,0}$ & $c_{s,1}$ & $c_{s,2}$ & $c_{s,3}$ & $c_{s,4}$\\
\midrule
36 & $3606$ & $16314$ & $15558$ & $-19428$ & $-7308$\\
37 & $-434$ & $-1970$ & $-1896$ & $2362$ & $894$\\
38 & $-1776$ & $-8018$ & $-7630$ & $9546$ & $3600$\\
39 & $1034$ & $4688$ & $4488$ & $-5590$ & $-2100$\\
40 & $-2042$ & $-9224$ & $-8778$ & $10984$ & $4134$\\\hline
\addlinespace[2pt]
41 & $25455$ & $115005$ & $109548$ & $-136959$ & $-51507$\\
42 & $5456$ & $24648$ & $23472$ & $-29336$ & $-11040$\\
43 & $-13750$ & $-62116$ & $-59136$ & $73970$ & $27816$\\
44 & $866$ & $3908$ & $3708$ & $-4642$ & $-1740$\\
45 & $7350$ & $33190$ & $31588$ & $-39530$ & $-14838$\\\hline
\addlinespace[2pt]
46 & $3965$ & $17903$ & $17040$ & $-21313$ & $-8013$\\
47 & $-1912$ & $-8716$ & $-9372$ & $10964$ & $4200$\\
48 & $-356$ & $-1544$ & $-1728$ & $1996$ & $768$\\
49 & $-48$ & $-112$ & $-72$ & $136$ & $48$\\
50 & $1010$ & $4526$ & $4432$ & $-5450$ & $-2058$\\\hline
\addlinespace[2pt]
51 & $-5508$ & $-24894$ & $-23796$ & $29700$ & $11178$\\
52 & $-76$ & $-292$ & $-264$ & $356$ & $132$\\
53 & $-32$ & $-88$ & $-8$ & $80$ & $24$\\
54 & $308$ & $1372$ & $1348$ & $-1648$ & $-624$\\
55 & $-1096$ & $-4940$ & $-4684$ & $5884$ & $2208$\\\hline
\addlinespace[2pt]
56 & $1970$ & $8948$ & $8544$ & $-10654$ & $-4008$\\
57 & $25203$ & $113837$ & $108376$ & $-135543$ & $-50967$\\
58 & $-8818$ & $-39910$ & $-38088$ & $47570$ & $17898$\\
59 & $27627$ & $124833$ & $118872$ & $-148647$ & $-55899$\\
60 & $-920$ & $-4208$ & $-4128$ & $5080$ & $1920$\\\hline
\addlinespace[2pt]
61 & $-32$ & $-128$ & $-72$ & $136$ & $48$\\
62 & $92$ & $416$ & $432$ & $-508$ & $-192$\\
63 & $-152$ & $-680$ & $-612$ & $796$ & $300$\\
64 & $-40$ & $-160$ & $-136$ & $192$ & $72$\\
65 & $490$ & $2206$ & $2088$ & $-2618$ & $-978$\\\hline
\addlinespace[2pt]
66 & $-1590$ & $-7170$ & $-6816$ & $8526$ & $3222$\\
67 & $634$ & $2878$ & $2796$ & $-3446$ & $-1302$\\
68 & $2638$ & $11906$ & $11322$ & $-14160$ & $-5328$\\
69 & $-1516$ & $-6886$ & $-6624$ & $8240$ & $3102$\\
70 & $3028$ & $13684$ & $13026$ & $-16286$ & $-6126$\\
\bottomrule
\end{tabular*}
\end{minipage}
\par\endgroup
\clearpage
\subsection*{Polynomial coefficients, $w_{71}$--$w_{139}$}
\begingroup\small\setlength{\parskip}{4pt}
Each row specifies $w_s(T)=c_{s,0}+c_{s,1}T+c_{s,2}T^2+c_{s,3}T^3+c_{s,4}T^4$.
\par\medskip\endgroup
\begingroup
\fontsize{9}{11}\selectfont
\setlength{\tabcolsep}{2pt}
\renewcommand{\arraystretch}{1.02}
\noindent
\begin{minipage}[t]{.49\linewidth}
\vspace{0pt}
\centering
\begin{tabular*}{\linewidth}{@{\extracolsep{\fill}}rrrrrr@{}}
\toprule
$s$ & $c_{s,0}$ & $c_{s,1}$ & $c_{s,2}$ & $c_{s,3}$ & $c_{s,4}$\\
\midrule
71 & $-1288$ & $-5818$ & $-5508$ & $6920$ & $2598$\\
72 & $-10910$ & $-49262$ & $-46880$ & $58662$ & $22050$\\
73 & $-5887$ & $-26575$ & $-25278$ & $31637$ & $11895$\\
74 & $292$ & $1348$ & $1284$ & $-1592$ & $-600$\\
75 & $4132$ & $18662$ & $17768$ & $-22208$ & $-8358$\\\hline
\addlinespace[2pt]
76 & $6184$ & $27928$ & $26746$ & $-33330$ & $-12546$\\
77 & $1986$ & $8972$ & $8608$ & $-10710$ & $-4032$\\
78 & $-5322$ & $-24108$ & $-23034$ & $28752$ & $10818$\\
79 & $-37725$ & $-170457$ & $-162414$ & $203043$ & $76365$\\
80 & $-8088$ & $-36544$ & $-34792$ & $43520$ & $16368$\\\hline
\addlinespace[2pt]
81 & $20384$ & $92114$ & $87720$ & $-109684$ & $-41250$\\
82 & $1312$ & $5992$ & $6000$ & $-7280$ & $-2760$\\
83 & $64$ & $280$ & $168$ & $-272$ & $-96$\\
84 & $-136$ & $-580$ & $-600$ & $728$ & $276$\\
85 & $232$ & $1060$ & $948$ & $-1220$ & $-456$\\\hline
\addlinespace[2pt]
86 & $80$ & $304$ & $232$ & $-328$ & $-120$\\
87 & $-734$ & $-3302$ & $-3108$ & $3922$ & $1470$\\
88 & $2370$ & $10674$ & $10116$ & $-12678$ & $-4770$\\
89 & $-926$ & $-4178$ & $-4080$ & $5038$ & $1902$\\
90 & $-3922$ & $-17696$ & $-16798$ & $21052$ & $7914$\\\hline
\addlinespace[2pt]
91 & $2210$ & $10052$ & $9756$ & $-12058$ & $-4548$\\
92 & $-4484$ & $-20282$ & $-19314$ & $24154$ & $9084$\\
93 & $1910$ & $8660$ & $8208$ & $-10282$ & $-3864$\\
94 & $16198$ & $73130$ & $69572$ & $-87050$ & $-32730$\\
95 & $8747$ & $39479$ & $37524$ & $-46975$ & $-17661$\\\hline
\addlinespace[2pt]
96 & $884$ & $4064$ & $4116$ & $-4960$ & $-1884$\\
97 & $-12$ & $-36$ & $-96$ & $84$ & $36$\\
98 & $136$ & $604$ & $600$ & $-728$ & $-276$\\
99 & $4$ & $-12$ & $-32$ & $28$ & $12$\\
100 & $-654$ & $-2962$ & $-2860$ & $3554$ & $1338$\\\hline
\addlinespace[2pt]
101 & $2258$ & $10194$ & $9704$ & $-12130$ & $-4566$\\
102 & $32$ & $128$ & $96$ & $-136$ & $-48$\\
103 & $-212$ & $-968$ & $-948$ & $1168$ & $444$\\
104 & $290$ & $1298$ & $1212$ & $-1534$ & $-570$\\
105 & $-1348$ & $-6106$ & $-5844$ & $7292$ & $2742$\\
\bottomrule
\end{tabular*}
\end{minipage}
\hfill%
\begin{minipage}[t]{.49\linewidth}
\vspace{0pt}
\centering
\begin{tabular*}{\linewidth}{@{\extracolsep{\fill}}rrrrrr@{}}
\toprule
$s$ & $c_{s,0}$ & $c_{s,1}$ & $c_{s,2}$ & $c_{s,3}$ & $c_{s,4}$\\
\midrule
106 & $-5443$ & $-24577$ & $-23390$ & $29269$ & $11001$\\
107 & $3196$ & $14464$ & $13806$ & $-17234$ & $-6486$\\
108 & $-6265$ & $-28297$ & $-26934$ & $33695$ & $12669$\\
109 & $-292$ & $-1324$ & $-1284$ & $1592$ & $600$\\
110 & $60$ & $288$ & $336$ & $-372$ & $-144$\\\hline
\addlinespace[2pt]
111 & $-44$ & $-188$ & $-168$ & $220$ & $84$\\
112 & $458$ & $2078$ & $1992$ & $-2482$ & $-930$\\
113 & $-584$ & $-2630$ & $-2496$ & $3124$ & $1182$\\
114 & $622$ & $2818$ & $2700$ & $-3362$ & $-1266$\\
115 & $-168$ & $-780$ & $-780$ & $948$ & $360$\\\hline
\addlinespace[2pt]
116 & $-1300$ & $-5878$ & $-5604$ & $7004$ & $2634$\\
117 & $378$ & $1722$ & $1656$ & $-2058$ & $-774$\\
118 & $-1240$ & $-5608$ & $-5844$ & $6956$ & $2652$\\
119 & $36$ & $36$ & $96$ & $-84$ & $-36$\\
120 & $-212$ & $-896$ & $-864$ & $1084$ & $408$\\\hline
\addlinespace[2pt]
121 & $-36$ & $-76$ & $24$ & $52$ & $12$\\
122 & $962$ & $4334$ & $4208$ & $-5202$ & $-1962$\\
123 & $-3354$ & $-15134$ & $-14388$ & $18014$ & $6774$\\
124 & $-64$ & $-256$ & $-168$ & $272$ & $96$\\
125 & $304$ & $1384$ & $1380$ & $-1676$ & $-636$\\\hline
\addlinespace[2pt]
126 & $-442$ & $-1978$ & $-1824$ & $2330$ & $870$\\
127 & $1982$ & $8984$ & $8640$ & $-10738$ & $-4044$\\
128 & $8081$ & $36483$ & $34712$ & $-43429$ & $-16329$\\
129 & $-4712$ & $-21350$ & $-20430$ & $25474$ & $9588$\\
130 & $9293$ & $41981$ & $39960$ & $-49981$ & $-18795$\\\hline
\addlinespace[2pt]
131 & $428$ & $1928$ & $1884$ & $-2320$ & $-876$\\
132 & $-72$ & $-324$ & $-432$ & $456$ & $180$\\
133 & $76$ & $316$ & $264$ & $-356$ & $-132$\\
134 & $-670$ & $-3046$ & $-2940$ & $3650$ & $1374$\\
135 & $874$ & $3928$ & $3708$ & $-4658$ & $-1752$\\\hline
\addlinespace[2pt]
136 & $-914$ & $-4142$ & $-3984$ & $4954$ & $1866$\\
137 & $228$ & $1068$ & $1116$ & $-1320$ & $-504$\\
138 & $1922$ & $8696$ & $8304$ & $-10366$ & $-3900$\\
139 & $-546$ & $-2502$ & $-2436$ & $3006$ & $1134$\\
\bottomrule
\end{tabular*}
\end{minipage}
\par\endgroup

\end{document}